\documentclass[aop]{imsart}

\usepackage[utf8]{inputenc}
\usepackage{amsmath}
\usepackage{amsthm}
\usepackage{amssymb}

\usepackage{enumitem}
\usepackage{placeins}
\usepackage{graphicx}
\usepackage{float}
\usepackage{caption}
\usepackage{subcaption}
\usepackage{ amssymb }
\usepackage{epstopdf}
\usepackage[utf8]{inputenc}
\usepackage[english]{babel}
\usepackage{ dsfont }
\usepackage[square,numbers,sort&compress]{natbib}
\usepackage{amsfonts}
\usepackage{scalerel}[2016/12/29]
\usepackage{bbm}
\theoremstyle{plain}
\newtheorem{theorem}{Theorem}[section]
\newtheorem{corollary}[theorem]{Corollary}

\newtheorem{prop}[theorem]{Proposition}
\newtheorem{lemma}[theorem]{Lemma}

\theoremstyle{definition}
\newtheorem{definition}[theorem]{Definition}

\theoremstyle{plain}

\newcommand{\nat}{\mathbb{N}}

\newcommand{\N}{\mathbb{N}}

\newcommand{\prob}{\mathbb{P}}
\newcommand{\p}{\mathbb{P}}
\newcommand{\E}{\mathbb{E}}
\newcommand{\expt}{\mathbb{E}}

\newcommand{\floor}[1]{{\left\lfloor #1 \right\rfloor}}
\newcommand{\ceil}[1]{{\left\lceil #1 \right\rceil}}

\newcommand{\sset}{\subset}

\newcommand{\la}{\lambda}
\newcommand{\al}{\alpha}

\newcommand{\mathforall}{\text{ for all }}

\newcommand{\mathand}{\;\text{and}\;}
\newcommand{\mathfor}{\;\text{for}\;}
\newcommand{\mathor}{\;\text{or}\;}

\newcommand{\mathas}{\;\text{as}\;}

\newcommand{\mathif}{\;\text{if}\;}

\newcommand{\ga}{\gamma}

\newcommand{\ep}{\epsilon}

\newcommand{\ka}{\kappa}
\newcommand{\de}{\delta}

\newcommand{\sig}{\sigma}

\newcommand{\del}{\partial}

\newcommand{\scrL}{\mathcal{L}}

\newcommand{\scrG}{\mathcal{G}}

\newcommand{\scrC}{\mathcal{C}}

\newcommand{\scrR}{\mathcal{R}}
\newcommand{\scrA}{\mathcal{A}}
\newcommand{\scrH}{\mathcal{H}}

\newcommand{\scrB}{\mathcal{B}}

\newcommand{\scrF}{\mathcal{F}}

\newcommand{\card}[1]{\left\vert #1 \right\vert}

\newcommand{\Z}{\mathds{Z}}

\newcommand{\R}{\mathds{R}}

\usepackage{MnSymbol}

\usepackage{yfonts}
\usepackage{ wasysym }

\newcommand{\eqd}{\stackrel{d}{=}}

\newcommand{\X}{\times}

\newcommand{\lf}{\left}
\newcommand{\rg}{\right}

\DeclareMathOperator{\Ai}{Ai}
\DeclareMathOperator{\Var}{Var}

\usepackage{parskip}

\begin{document}
	
	\begin{frontmatter}
		
		\title{Bulk properties of the Airy line ensemble}
		\runtitle{The Airy line ensemble}
		
		\begin{aug}
			\author{\fnms{Duncan} \snm{Dauvergne}\ead[label=e1]{dd18@math.princeton.edu}}
			\thanks{D.D. was supported by an NSERC CGS D scholarship.}
			\and
			\author{\fnms{B\'alint} \snm{Vir\'ag}\ead[label=e2]{balint@math.toronto.edu}}
			\thanks{B.V. was supported by the Canada Research Chair program, the NSERC Discovery Accelerator grant, the MTA Momentum Random Spectra research group, and the ERC consolidator grant 648017 (Abert).}
			
			\runauthor{Dauvergne and Vir\'ag}
			
			\address[A]{Department of Mathematics, Princeton University \\
				\printead{e1}\\}
			
			\address[B]{Department of Mathematics, University of Toronto \\
				\printead{e2}\\}
		\end{aug}
		
		\begin{abstract}: The Airy line ensemble is a central object in random matrix theory and last passage percolation defined by a determinantal formula. The goal of this paper is to provide a set of tools which allow for precise probabilistic analysis of the Airy line ensemble. The two main theorems are a representation in terms of independent Brownian bridges connecting a fine grid of points, and a modulus of continuity result for all lines. Along the way, we give tail bounds and moduli of continuity for nonintersecting Brownian ensembles, and a quick proof of tightness for Dyson's Brownian motion converging to the Airy line ensemble.
		\end{abstract}
		
		
		
	\end{frontmatter}

	\section{Introduction}
	
	The Airy line ensemble is a central object in random matrix theory, last passage percolation, and more generally, for problems about the Kardar-Parisi-Zhang universality class. It was first described by Pr\"ahofer and Spohn~\citep{prahofer2002scale} as the scaling limit of the polynuclear growth model, see Section~\ref{s:related} for more details on its history and related work.
	
	\subsection*{A brief description} The \textbf{parabolic Airy line ensemble} $\scrL$ is a decreasing sequence of nonintersecting continuous functions $\scrL_1 > \scrL_2 > \scrL_3 \dots$  where each $\scrL_i:\R \to \R$. It is the unique process of nonintersecting continuous functions whose finite dimensional distributions form a determinantal process with kernel~\eqref{E:airy-kernel}.
	
	
	We use the term parabolic in front of Airy line ensemble to help distinguish the object from $\scrL(t) + t^2$, which is known as the \textbf{Airy line ensemble}. The process $\scrL(t) + t^2$ is stationary. Also, for any fixed $t$, the distribution of $\scrL_1(t) + t^2$ is a GUE Tracy-Widom random variable (see~\cite{prahofer2002scale}) and
	\begin{equation}
	\label{E:asymptotics}
	\scrL_k(t) + t^2 \sim -(3\pi k/2)^{2/3} \qquad \text{ as } k\to \infty,
	\end{equation}
	see Lemma~\ref{L:var-airy} and Corollary~\ref{C:point-locations} for more precise asymptotics.
	The determinantal formula~\eqref{E:airy-kernel} is useful for proving convergence to $\scrL$ and for proving some properties of fixed-time distributions. However, it is hard to deduce even the most basic path properties, such as continuity, from it directly: see~\citep{prahofer2002scale}, Appendix A for the essential steps of a proof of continuity using just the determinantal formula.

	
	A useful technique, called the Brownian Gibbs property, was developed by Corwin and Hammond~\citep{CH}. The Brownian Gibbs property says that inside any region, conditionally on the outside of the region, the parabolic Airy line ensemble is just a sequence of independent Brownian bridges of variance $2$ conditioned so that everything remains nonintersecting and continuous.
	
	
	The Brownian Gibbs property implies that if the boundary of a region is well understood, then one can use properties of Brownian bridges to deduce path properties of $\scrL$. This is a big if: for a rectangular region, left and right boundary points can be jammed close together, making nonintersection difficult. Also, the bottom and top boundaries are paths, whose properties are not easily accessible from the determinantal structure. These problems are particularly difficult to tackle further into the large-$k$ bulk of the Airy line ensemble, where lines are more closely spaced and the top boundary cannot be easily removed.
	The goal of this paper is to tackle these issues in order to make the Airy line ensemble more amenable to probabilistic analysis.
	
	\subsection*{Two basic tools}
	For any random object, the most fundamental tool for studying it is a construction with a rich structure of independence. In this paper we obtain such a construction -- the bridge representation -- that quantitatively relates the parabolic Airy lines to independent Brownian bridges. We also obtain tight control of the parabolic Airy line locations and exponential moment bounds for the number of lines that are close together at a given time, which makes the bridge representation useful in practice.
	
	
	A second basic tool, which for the case of Brownian motion is often the first theorem in an introductory textbook, is a good modulus of continuity bound. By using the bridge representation, we obtain modulus of continuity bounds for parabolic Airy lines that are optimal up to a logarithmic factor in the number of lines.
	
	
	The bridge representation and the strong modulus of continuity for parabolic Airy lines are crucial steps in the construction of the scaling limit of Brownian last passage percolation, known as the directed landscape~\citep{DOV}. In that paper, having a detailed and quantitative understanding of the parabolic Airy line ensemble in the bulk is essential for estimating a certain last passage value across $\scrL$. The upcoming work~\cite{DNV} also relies on the bridge representation to show that other models in the KPZ universality class converge to the directed landscape.
	
	\subsection*{The bridge representation} The Brownian Gibbs property suggests that one could construct the top $k$ lines of $\scrL$ by sampling points on a fine space-time grid, then connecting them with independent Brownian bridges that will not intersect because of the fineness of the grid.
	Indeed, we have such a result, with one difference: when a group of endpoints are close together, we have to condition the Brownian bridges between those endpoints not to intersect. However, we have good control over the size of these groups of close endpoints. In particular, they will remain bounded as we include more and more lines in the scales that we are working with. The close endpoint phenomenon is not a deficiency in our method; close endpoints really do exist in the parabolic Airy line ensemble at these scales.
	
	
	To make this more precise, pick parameters $\ell, k \in \N$ and $t, \de > 0$. Let $s_j=tj/\ell$, and sample $\scrL$ at grid points $\scrL_i(s_j)$ for $i \in\{1,\dots, 2k\}$ and $j \in \{0,\dots,\ell\}$. The bridges connecting these points will be indexed by
	$$
	(i, j) \in S := \{1,\dots, 2k\} \X \{1,\dots,\ell\}.
	$$
	Let $G$ be the random graph on $S$ that connects $(i, j)$ to $(i+1, j)$
	if
	$$
	\min(|\scrL_i(s_{j-1}) - \scrL_{i+1}(s_{j-1})|, |\scrL_i(s_{j}) - \scrL_{i+1}(s_{j})|) \le \de.
	$$
	Now for each $i$, connect up the points $\{\scrL_i(s_j)\}$ with independent, variance $2$, Brownian bridges $B_{i,j}:[s_{j-1},s_j]\to\mathbb R$, where $B_{i, j}$ and $B_{i + 1, j}$ are conditioned not to intersect whenever $(i, j)$ and $(i+1, j)$ are in the same component of $G$. This yields a new line ensemble $\mathcal B$, see Figure~\ref{fig:Bridge} for an illustration.
	
	
	
	
	We call $\scrB$ a bridge representation for $\scrL$. For the process $\scrB$ to mimic $\scrL$, it must be nonintersecting with high probability. For this, by Brownian scaling, we need that $\de \gg  \sqrt{t/\ell}$ or else the Brownian bridges coming from different components of $G$ will intersect with nonnegligible probability.
	
	
	Also, regardless of the parameter choices we should not expect that $\scrB|_{\{1, \dots, 2k\} \X [0, t]}$ and $\scrL|_{\{1, \dots, 2k\} \X [0, t]}$ are close in law since $\scrB$ ignores the lower boundary $\scrL_{2k + 1}$ entirely. However, if we have chosen the grid finely enough so that edges in $G$ are sparse (i.e. most bridges $B_{i, j}$ are unconditioned Brownian bridges), then the effect of this lower boundary should be localized to the lowest lines. In particular, it is reasonable to expect that in the right parameter range, the top $k$ lines of $\scrB$ and $\scrL$ are close in law.
	
	
	The distance between the $k$th and $(k+1)$st Airy points is of order $k^{-1/3}$ (see Equation~\eqref{E:asymptotics}), so for edges in $G$ to be sparse we need $\de = o(k^{-1/3})$. Therefore by the rationale discussed above, the mesh size $t/\ell$ should be $o(k^{-2/3})$. For this reason, we use a parameter $\gamma > 0$ so that the mesh size $t/\ell$ is at most $k^{-2/3-\gamma}$.
	
	\begin{theorem}\label{T:bridge-rep-i}
		Let $\ga \in [c_2 \log(\log k)/\log k, 2]$, and let
		$$k \ge 3,\qquad t > 0, \qquad \de = k^{-1/3 - \ga/4}, \qquad \ell \ge tk^{2/3 + \ga}.$$
		The total variation distance between $\scrB$ and $\scrL$ restricted to the top $k$ lines and the interval $[0,t]$ satisfies
		$$
		d_{TV}(\scrB|_{\{1, \dots, k\}\X[0, t]}, \scrL|_{\{1, \dots, k\}\X[0, t]})\le \ell \exp(-c_1 \ga k^{\ga/12}).
		$$
		Here $c_1$ and $c_2$ are positive universal constants.
	\end{theorem}
	
	The possibility of extreme events involving the Airy points or the Brownian bridges is what forces us to take $k^\ga$ and $\de/\sqrt{t/\ell}$ to be at least a power of $\log k$ in Theorem~\ref{T:bridge-rep-i}. We have not attempted to optimize the particular power.
	Note that the use of $2k$ lines in the bridge representation is chosen for convenience; a smaller value could be used without affecting the statement of Theorem~\ref{T:bridge-rep-i}.
	
	
	In practical applications, one has to choose the value $\gamma$ judiciously. A small value gives more conflicts and larger groups of multiple paths, while a large value means more grid points need to be controlled.
	
	
	\begin{figure}%
		\centering
		\begin{subfigure}[t]{5cm}
			\includegraphics[width=5cm]{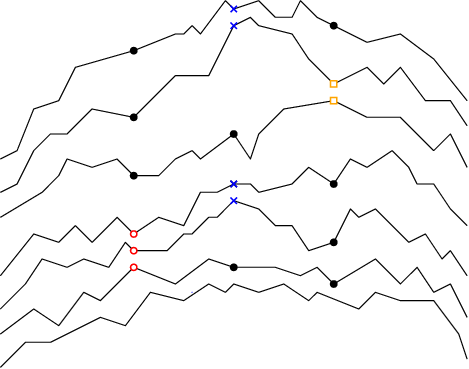}
			\caption{}
		\end{subfigure}
		\qquad
		\begin{subfigure}[t]{5cm}
			\includegraphics[width=5cm]{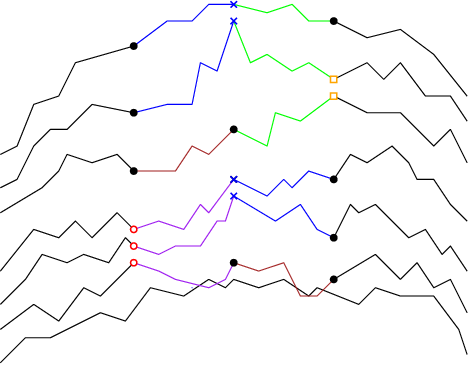}
			\caption{}
		\end{subfigure}
		\caption{An illustration of the bridge representation $\scrB$ of the parabolic Airy line ensemble. Figure~\ref{fig:Bridge}(a) is the process $\scrL$, with points at three times identified. Points with the same time coordinate are grouped together if they are close (i.e. have the same colour/shape). To sample the bridge representation on this grid, we erase all lines between the specified points and resample independent Brownian bridges that are conditioned not to intersect each other if either of their endpoints are close. The result is Figure~\ref{fig:Bridge}(b). Because we ignored the lower boundary condition, there is a reasonable chance that the lowest bridges will intersect this boundary. However, if we only look at the top half of the lines in the bridge representation, then with high probability they do not intersect each other and resemble what we would have obtained from the usual Brownian Gibbs resampling.}%
		\label{fig:Bridge}%
	\end{figure}
	
	We have good control over the behavior of the graph $G$ used to construct the bridge representation. In particular, the probability that a typical vertex is not isolated is $O(k^{-3\ga/4})$ and edges in $G$ are typically well spaced (see Proposition~\ref{P:edge-spread} for a precise statement). Also, the maximal component size $M_k$ of the graph $G$ behaves well; our general bound in the most common case of fixed $\ga$ gives the following.
	\begin{prop}
		\label{P:delta-chains-i}
		Fix $\gamma>0$, $t>0$ and let $\ell=\ell_k = \ceil{tk^{2/3 + \ga}}$ and $\de = \de_k = k^{-1/3 - \ga/4}$. Then
		$$
		\prob (M_k \ge 14(1+1/\gamma)) \to 0 \qquad \mathas k \to \infty.
		$$
	\end{prop}

	One of the main technical difficulties in proving Theorem~\ref{T:bridge-rep-i} is understanding the distribution of pairs of points that are within $\de$ of each other in the Airy point process $\{\scrL_k(0) : k \in \N\}$. In an interval of the form $[-a,-a+q]$, where $a\ge 1$ and $q$ is either smaller or of comparable size to $a$, the number $L$ of such pairs is typically of the order $\eta =a^2\de^3q$. By analyzing the determinantal structure of the Airy point process we show that the probability of being larger than this typical value decays exponentially. We state the following bound for all $a \ge 1$ and $q \ge 0$. Note that it is far from optimal when $q \gg a$.
	
	\begin{prop}\label{P:mean-conc-i} There exists a universal constant $c>0$ such that for all $a \ge 1$ and $q \ge 0$, we have
		$$
		\prob (L > m \eta) \le 2e^{-c m\lf(\frac{\eta}{\eta + 1}\rg)} \qquad \text{ for all }  m>0.
		$$
	\end{prop}
	
	Similar results hold for pairs of close points in all determinantal processes with nice enough kernels, see Proposition~\ref{P:de-jam-gen}.

	\subsection*{Modulus of continuity and a law of large numbers}
	
	The bridge representation yields a modulus of continuity bound for $\scrL$ that is optimal up to the power of $\log k$.
	
	\begin{theorem}\label{T:mod-cont-i}
		There is a constant $d > 0$ so that for all $t>0$ there is a random $C>0$ so that a.s.
		$$
		\frac{|\scrL_k(s+r) - \scrL_k(s)|}{\sqrt{r\log(1+1/r)} \log^d k} < C
		$$
		for all $k\in \N$, $r>0$, $s,s+r \in[0,t]$.
	\end{theorem}
	
	We have not tried to extract explicit tail bounds on the constant $C$ in Theorem~\ref{T:mod-cont-i}. Along the way to proving the bridge representation, we also show a preliminary modulus on individual lines that does yield good tail bounds. 
	
	\begin{theorem}
		\label{T:airy-mod-i}
		There are $c,d>0$ so that for any $k \in \N$, there exists a random constant $C_k$ so that
		$$
		|\scrL_k(s) - \scrL_k(s + r)| \le C_k \sqrt{r} \log^{1/2}(2/r)
		$$
		for all $r > 0, s, s + r \in [0, 1]$. Moreover,
		$$
		\p(C_k > m) \le e^{ck - d m^2}  \qquad \text{for all }
		m > 0.$$
	\end{theorem}
	
	Up to the constant $d$ and the prefactor $e^{ck}$, the modulus of continuity in Theorem~\ref{T:airy-mod-i} is the same as that of a single Brownian motion without any nonintersection conditions. Theorem~\ref{T:airy-mod-i} immediately implies a statement in the form of Theorem~\ref{T:mod-cont-i} with $\log^d(k)$ replaced by $\sqrt{k}$. Reducing this $k$-dependent factor requires the full power of the bridge representation.
	
	Note that both Theorem~\ref{T:mod-cont-i} and Theorem~\ref{T:airy-mod-i} can be used to give modulus of continuity bounds over other intervals not anchored at $0$ by using stationarity of $\scrL(t) + t^2$.
	
	Many authors have previously studied continuity properties of $\scrL$. Pr\"ahofer and Spohn outlined the essential steps of a proof of continuity in the appendix of~\cite{prahofer2002scale}. Corwin and Hammond~\cite{CH} showed that all parabolic Airy lines are locally absolutely continuous with respect to Brownian motion. Hammond~\citep{hammond2016brownian} found explicit modulus of continuity bounds for more general line ensembles as part of a broader program. In the case of the parabolic Airy line ensemble these bounds grow exponentially in $k$. 
	
	After this paper was first posted, Calvert, Hammond, and Hegde~\citep{calvert2019brownian}, Theorem 3.11, found explicit Radon-Nikodym derivative estimates for Airy lines versus Brownian motion. These estimates can be used to give a version of Theorem~\ref{T:airy-mod-i} where the random constants $C_k$ satisfy $\p(C_k > m) \le e^{c_k -dm^2}$ with the optimal value of $d = 1/4$ and a sequence of exponentially growing constants $c_k$.
	
	
	A good modulus of continuity bound can be naturally combined with pointwise bounds to give a uniform law of large numbers for the parabolic Airy line ensemble.
	
	\begin{theorem}
		\label{T:lln-i}
		For any sequences $\ell_k, t_k>0$ with $\log k = o(\ell_k)$ and $t_k = e^{o(\ell_k)}$, we have
		$$
		\lim_{k \to \infty} \sup_{s \in [-t_k, t_k]}\frac{k^{1/3}|{\mathcal L}_k(s)+s^2+(3  \pi k/2)^{2/3}|}{\ell_k} = 0 \qquad \text{ almost surely.}
		$$
	\end{theorem}
	
	Theorem~\ref{T:lln-i} only requires Theorem~\ref{T:airy-mod-i}, rather than the more precise bound in Theorem~\ref{T:mod-cont-i}.

	\subsection*{Properties of nonintersecting Brownian ensembles} On the way to proving these results, we establish several properties of Brownian bridges conditioned not to intersect. A modulus of continuity bound can be
	directly deduced from the following, see Proposition~\ref{P:bridge-bd}.
	\begin{prop} \label{P:nonint-dev-i}
		Let $B_i$, $i\in\{1,\ldots k\}$ be independent Brownian bridges
		with slope $b_i$ on some time interval $[a,b]$ conditioned not to intersect. Then for any $i$, $t>0$, $s,s+t \in [a,b]$ we
		have
		$$
		\p(|B_i(s+t)-B_i(s)-b_it|\ge m\sqrt{t})\le e^{ck-dm^2}.
		$$
		Here $c, d > 0$ are universal constants.
	\end{prop}
	The $k$-dependent factor $e^{ck}$ is sharp up to the value of the constant $c$.
	
	
	By combining Proposition~\ref{P:nonint-dev-i} with a stochastic domination argument and one-point bounds for the top lines in a Dyson's Brownian motion $\{W^n_1 > W^n_2 > \dots > W^n_n\}$ (i.e. a collection of $n$ Brownian motions all started at $0$ conditioned not to intersect on $(0, \infty)$), we get a short proof of a tail bound for Dyson's Brownian motion at the edge. This next proposition was originally obtained by Hammond~\cite{hammond2016brownian}, see Theorem 2.14.
	
	\begin{prop}\label{P:Dyson-tails-i}
		Fix $k\in \N,c > 0$. For every $n \in \N$, $t > 0, \;s \in (0, ctn^{-1/3}]$, and $m > 0$ we have
		$$
		\p\bigg(\lf| W_k^n(t +s)- W_k^n(t) - s\sqrt{n/t}\rg| > m \sqrt{s} \bigg) \le c_k e^{-d_k m^{3/2}}.
		$$
		Here $c_k, d_k$ are $k$-dependent constants.
	\end{prop}
	
	Significantly, the bound in Proposition~\ref{P:Dyson-tails-i} does not get worse with $n$ at the scale where Dyson's Brownian motion converges to the parabolic Airy line ensemble. Therefore together with the Kolmogorov-Centsov criterion and convergence of the finite-dimensional distributions to those of the parabolic Airy line ensemble, Proposition~\ref{P:Dyson-tails-i} immediately implies functional tightness of Dyson's Brownian motion at the edge.
	
	\begin{theorem}
		\label{T:CH-i}
		Define
		$$
		\scrC^n_k(t) = \lf(W^n_k(1 + 2tn^{-1/3}) - 2 \sqrt{n} - 2tn^{1/6} \rg)n^{1/6}.
		$$
		For every $k$, $\{\scrC^n_k, n\ge 1\}$ is tight with respect to the uniform-on-compact topology.
	\end{theorem}
	
	Theorem~\ref{T:CH-i} was first obtained by \cite{CH}.
	Finally, we note that Proposition~\ref{P:Dyson-tails-i} can be used to give a type of law of the iterated logarithm for Dyson's Brownian motion, see Proposition~\ref{P:cross-prob}.
	\subsection{Related work}\label{s:related}
	
	The main tools we use were introduced by Corwin and Hammond~\citep{CH}. They are
	the Brownian Gibbs property and a monotonicity lemma for nonintersecting Brownian bridges with respect to their endpoints, see Lemma~\ref{L:monotone-gibbs}.
	
	
	Ergodicity of the Airy line ensemble for time-shifts was proven by Corwin and Sun~\citep{corwin2014ergodicity}. Hammond~\citep{hammond2016brownian} used the Brownian Gibbs property to prove Radon-Nikodym derivative and other regularity bounds for the parabolic Airy line ensemble with respect to Brownian bridges. 
	Calvert, Hammond, and Hegde~\citep{calvert2019brownian} built on the results of~\citep{hammond2016brownian} to give Radon-Nikodym derivative bounds for parabolic Airy lines with respect to Brownian motion. Both the results of~\citep{hammond2016brownian} and~\citep{calvert2019brownian} apply to more general line ensembles.
	Hammond~\citep{hammond2019modulus, hammond2020exponents, hammond2019patchwork} applied the results in ~\citep{hammond2016brownian} to understanding problems about the geometry of last passage paths in Brownian last passage percolation and the roughness of limiting growth profiles in that model. 
	
	The work of Hammond~\citep{hammond2016brownian} and Calvert, Hammond, and Hegde~\citep{calvert2019brownian} is tailored to studying properties of individual Airy lines at the edge, and individual lines in more general Brownian ensembles. In particular, many of the highlights of these works concern path properties of the Airy process $\scrL_1(t) + t^2$. In comparison, our work is focused on understanding the joint behaviour of large regions of the Airy line ensemble, and on controlling the behaviour of lines deeper in the bulk.

	
	One goal of much of the above work is to characterize the Airy line ensemble without relying as much on determinantal formulas.  Such a characterization could allow a method for proving convergence results for models within the KPZ universality class that either have no exact formulas, or only have intractable formulas.
	In particular, Corwin and Hammond~\citep{CH} conjectured that, up to a value shift, the parabolic Airy line ensemble is the unique nonintersecting line ensemble satisfying both the Brownian Gibbs property and stationarity after the addition of a parabola. In~\citep{corwin2016kpz}, they outlined how this conjecture could be used to prove convergence of the KPZ equation to the Airy process. Progress towards proving this conjecture was made by Dimitrov and Matetski~\citep{dimitrov2020characterization}, who showed that the parabolic Airy line ensemble is uniquely characterized by the Brownian Gibbs property and the law of its top line. 
	
	
	Finally, our work and its application in~\cite{DOV} is part of a growing body of literature focussed on understanding last passage percolation and the KPZ universality class via probabilistic and geometric methods, rather than purely through analysis of exact formulas. In addition to the work discussed above, some prominent examples of this include Basu, Sidoravicius, and Sly's resolution of the slow bond problem~\cite{basu2014last}, see also~\cite{basu2017invariant}, Corwin and Hammond's study of the KPZ line ensemble~\cite{corwin2016kpz}, and work on understanding the structure of last passage geodesics, such as Basu, Hoffman, and Sly~\cite{basu2018nonexistence}; Basu, Sarkar, and Sly~\cite{basu2017coal}; Georgiou, Rassoul-Agha, and Sepp\"al\"ainen~\cite{georgiou2017geodesics}; Hammond and Sarkar~\cite{hammond2020modulus}; and Pimentel~\cite{pimentel2016duality}.

	\subsection*{Organization of the paper}
	The paper is organized as follows. In Section~\ref{S:preliminaries}, we define the most important terms. 
	
	In Section~\ref{S:nonintersecting} we prove a concentration result for Dyson's Brownian motion, and we prove increment tail bounds and a modulus of continuity result for nonintersecting Brownian bridges. Proposition~\ref{P:nonint-dev-i} is proven as Lemma~\ref{L:bridge-disc} within this section. 
	
	Section~\ref{S:mod-Dyson} gives tail bounds for increments of the top Dyson's Brownian motion lines.  Proposition~\ref{P:Dyson-tails-i} is restated and proven in this section as Proposition~\ref{P:dyson-tails}, and Theorem~\ref{T:CH-i} is proven immediately after that proposition.
	
	In Section~\ref{S:airy-point} we study the
	Airy point process 
	$
	\scrL_1(0) > \scrL_2(0) > \dots.
	$
	We recall and prove theorems about point locations and prove new results about close points. Proposition~\ref{P:delta-chains-i} follows as a special case of Proposition~\ref{P:delta-chains}, and Proposition~\ref{P:mean-conc-i} is restated and proven as Proposition~\ref{P:mean-conc}. 
	
	In Section~\ref{S:moduli} we establish a preliminary modulus of continuity bound on the parabolic Airy line ensemble and prove the law of large numbers, Theorem~\ref{T:lln-i}. Theorem~\ref{T:airy-mod-i} is proven in this section as a special case of Theorem~\ref{T:airy-mod}.
	
	In Section~\ref{S:bridge} we prove the bridge representation, Theorem~\ref{T:bridge-rep-i}. This theorem is restated and proven as Theorem~\ref{T:bridge-rep}. 
	
	Section~\ref{S:pr-bridge} contains the proof of the modulus of continuity result, Theorem~\ref{T:mod-cont-i}. We actually prove a slightly stronger result, Theorem~\ref{T:mod-cont}, from which Theorem~\ref{T:mod-cont-i} follows as an immediate corollary.
	
	
	Some of the technical results we prove here are stated in greater generality than we need in the paper (in particular, Lemma~\ref{L:levy-est}), or for a greater range of parameters. We do this for use in the paper~\cite{DOV} and in other future work.
	\section{Preliminaries}\label{S:preliminaries}
	
	In this section we recall definitions related to Dyson's Brownian motion and the parabolic Airy line ensemble.
	Let
	$$
	W^k = (W^k_1, \dots, W^k_k):[0, \infty) \to \R
	$$
	be $k$ independent Brownian motions started from the initial condition $W^k_i(0) = 0$ and conditioned not to intersect on $[\ep, t]$ for some $\ep, t > 0$. The limit as $\ep \to 0$ and $t \to \infty$ of this process exists (e.g. see \cite{grabiner1999brownian} and Section 2 in \cite{o2002random} for details regarding this), and is known as a $k$-level \textbf{Dyson's Brownian motion}.
	
	A $k$-level Dyson's Brownian motion has the same distribution as the eigenvalues of a matrix-valued Brownian motion in the space of $k \X k$ Hermitian matrices with complex entries. This is a result that goes back to Dyson \cite{dyson1962brownian}, see also \cite{grabiner1999brownian} for the connection with nonintersecting random walks. Note that just like with Brownian motion, the law of $W^k$ is invariant under Brownian scaling and time inversion. To see this, we could simply apply Brownian scaling and time inversion prior to taking the $\ep \to 0, t \to \infty$ limiting procedure above.
	
	A \textbf{Brownian $k$-melon}
	$$
	B^k = (B^k_1, \dots, B^k_k):[0, t] \to \R^k
	$$
	is a system of $k$ independent Brownian bridges $B_1, \dots, B_k$ with $B^k_i(0) = B^k_i(t) = 0$ for all $i \in \{1, \dots, k\}$, conditioned in a similar limiting fashion so that
	$$
	B^k_1(s) > B^k_2(s) > \dots > B^k_k(s)
	$$
	for all $s \in (0, t)$ (the reason for the name is that nonintersecting Brownian bridges look like stripes on a watermelon). We note that analogously to the usual relationship between Brownian bridge and Brownian motion, we have that
	\begin{equation}
	\label{E:melon-dyson}
	B^k(s) \eqd  \frac{t - s}{\sqrt{t}}W^k\lf(\frac{s}{t - s}\rg),
	\end{equation}
	where the equality above holds in distribution in the space of $k$-tuples of random continuous functions from $[0, t]$ to $\R$. This follows by relating nonintersecting Brownian motions to nonintersecting Brownian bridges before passing to the limit.
	
	
	More precisely, if $W^k$ is a collection of $n$ independent Brownian motions with initial conditions $W^k_i(0) = 0$ conditioned not to intersect on $[\ep, t_0]$, then the process $B^k$ defined from $W^k$ as in \eqref{E:melon-dyson} is a collection of $n$ Brownian bridges with $B^k_i(0) = B^k_i(t) = 0$, conditioned not to intersect on $[\ep t/(1+\ep), t_0 t/(1 + t_0)]$. As we take $\ep \to 0, t_0 \to \infty$, $W^k$ converges to a Dyson's Brownian motion and $B^k$ converges to a Brownian $k$-melon.
	
	
	We say that a Brownian motion or a Brownian bridge has \textbf{variance $v$} if its quadratic variation in an interval $[s,t]$ is equal to $v(t-s)$. We say that a Dyson's Brownian motion or a Brownian $k$-melon has variance $v$ if the component Brownian motions have variance $v$.
	
	
	The top lines of an $n$-level Dyson's Brownian motion (or alternatively, a Brownian $n$-melon) converge in law to a limit called the parabolic Airy line ensemble as $n \to \infty$. To state this convergence precisely, we first discuss line ensembles and define the limiting process.
	
	\subsection*{Line ensembles}
	
	\begin{definition}
		\label{D:line-ensemble} A \textbf{line ensemble} $\scrR = (\scrR_1, \scrR_2, \dots)$ is a possibly finite sequence of random functions where each $\scrR_i:I \to \R$ for some interval $I \sset \R$. We say that a line ensemble is \textbf{ordered} if almost surely,
		$$
		\scrR_i(x) \ge \scrR_{i+1}(x) \qquad \mathforall i \in \N, x \in \R.
		$$
		We say that $\scrR$ is \textbf{strictly ordered} if strict inequality can replace weak inequality above for all $i, x$.
	\end{definition}
	
	We write $\scrR|_{\{i, \dots, k\} \X [c, d]}$ for the sequence $(\scrR_i, \dots, \scrR_k)$ restricted to the interval $[c, d]$.
	
	\begin{definition}
		\label{D:brownian-gibbs} An ordered line ensemble $\scrL$ satisfies the \textbf{Brownian Gibbs property} (with variance $v$) if the following holds for all $k \in \{0, 1, \dots\}, \ell \in \N$ and $[c, d] \sset I$ for which $\scrR_{k+1}, \dots, \scrR_{k+\ell}$ are all defined. Let
		$\scrF$ be the $\sigma$-algebra generated by the set
		$$
		\{\scrR_i(x) : (i, x) \notin \{k + 1, \dots, k + \ell\} \X [c, d] \}.
		$$
		Then the conditional distribution of $\scrR|_{\{k + 1, \dots, k + \ell\} \X [c, d]}$ given $\scrF$ is equal to the law of $\ell$ independent, variance $v$, Brownian bridges $B_1, \dots, B_\ell: (c, d) \to \R$ with $B_i(c) = \scrL_{k+i}(c)$ and $B_{k+i}(d) = \scrR_i(d)$ for all $i \in \{1, \dots, \ell\}$, conditioned on the event
		$$
		\scrR_k(r) > B_1(r) > B_2(r) >\dots > B_\ell(r) > \scrR_{k+\ell+1}(r) \qquad \mathforall r \in (c, d).
		$$
		(Again, this conditioning should be understood in the same limiting sense as in the definition of Dyson's Brownian motion in the case when some endpoints are equal).
		If $\scrR_k$ or $\scrR_{k+\ell+1}$ does not exist, then drop the corresponding inequality from the conditioning.
	\end{definition}
	
	Rather than repeating the above statements throughout the paper to describe a sequence of Brownian bridges with the above properties, we will simply say that $B_1, \dots, B_\ell$ is a sequence of Brownian bridges with endpoints $B_i(c) = \scrR_{k+i}(c)$ and $B_i(d) = \scrR_{k+i}(d)$ conditioned to avoid each other and the boundaries $\scrR_k, \scrR_{k + \ell  + 1}$.
	
	
	Note that both Dyson's Brownian motion and Brownian $k$-melons are line ensembles with the Brownian Gibbs property.
	
	\begin{definition}
		The \textbf{parabolic Airy line ensemble} $\scrL = \{\scrL_i : i \in \N\}$ is a continuous, strictly ordered line ensemble with lines $\scrL_i:\R \to \R$ indexed by $\N$ defined by the requirement that the process $\scrA(t) = \scrL(t) + t^2 $ is a determinantal process with kernel
		\begin{equation}
		\label{E:airy-kernel}
		K((x, s) ; (y, t)) =
		\begin{cases}
		\int_{0}^\infty d \lambda e^{-\lambda(s - t)}\Ai(x + \la)\Ai(y + \la), \qquad & \mathif s \ge t, \\
		-\int_{-\infty}^0 d \lambda e^{-\lambda(s - t)}\Ai(x + \la)\Ai(y + \la), \qquad & \mathif s < t,
		\end{cases}
		\end{equation}
		where $\Ai(\cdot)$ is the Airy function. The process $\scrA(t) = \scrL(t) + t^2$ is stationary, and is known simply as the \textbf{Airy line ensemble}.
	\end{definition}
	
	See~\cite{hough2009zeros} for general background on determinantal processes, including the definition of a determinantal point process from a kernel. We note that the kernel formula \eqref{E:airy-kernel} implies that $\scrL(\cdot)$ is equal in distribution to its time reversal $\scrL(- \;\cdot)$.
	
	
	Prah\"ofer and Spohn~\citep{prahofer2002scale} first showed that a multi-line process with the kernel~\eqref{E:airy-kernel} arises a scaling limit of the polynuclear growth model. Adler and van Moerbeke~\citep{adler2005pdes} then showed that the finite dimensional distributions of a rescaled $n$-level Dyson's Brownian motion also converge to such a multiline process. Corwin and Hammond \citep{CH} then rigorously showed that this limit can be realized as a process of locally Brownian, nonintersecting functions. They also showed that  Dyson's Brownian motion converges to the parabolic Airy line ensemble uniformly on compact sets, and that the parabolic Airy line ensemble satisfies the Brownian Gibbs property. Note that the convergence in~\cite{CH} is proven for Brownian melons, rather than Dyson's Brownian motion. The two convergence statements are equivalent by~\eqref{E:melon-dyson}.
	To describe these theorems, we first introduce the scaling of Dyson's Brownian motion.
	
	
	Let $(W^n =(W^n_1, \dots, W^n_n))_{n \in \N}$ be a sequence of $n$-level Dyson's Brownian motions. For $k < n$, define
	\begin{equation}
	\label{E:rescaled-dyson}
	\scrC^n_k(t) = \lf(W^n_k(1 + 2tn^{-1/3}) - 2 \sqrt{n} - 2tn^{1/6} \rg)n^{1/6}.
	\end{equation}
	
	Let $\scrC^n$ be the line ensemble with $n$ lines whose $k$th line is given by $\scrC^n_k$. The line ensemble $\scrC^n$ has the Brownian Gibbs property since it is an affine shift of a Dyson's Brownian motion. Note that the variance of $C^n$ is $2$, rather than $1$, because of how time and space were rescaled.
	
	\begin{theorem} [\citep{adler2005pdes},~\citep{CH}]
		\label{T:airy-BG}
		The parabolic Airy line ensemble $\scrL$ is the distributional limit of the line ensembles $\scrC^n$ as $n \to \infty$ in the sense that for any set $\{1, \dots, k\} \X [c, d] \sset \Z \X \R$, the functions $\scrC^n|_{\{1, \dots, k\} \X [c, d]}$ converge in distribution to $\scrL|_{\{1, \dots, k\} \X [c, d]}$ in the topology of uniform convergence of $k$-tuples functions from $[c, d] \to \R$.
		Moreover, $\scrL$ has the Brownian Gibbs property with variance $2$ and is a strictly ordered line ensemble with probability $1$.
	\end{theorem}

	\begin{figure}%
		\centering
		\includegraphics[width=6cm]{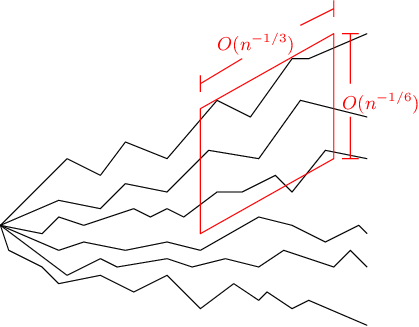}
		\caption{Dyson's Brownian motion from time $0$ to time $1$. If we zoom in around the location $(1, 2\sqrt{n})$ on an $O(n^{-1/3})\times O(n^{-1/6})$ parallelogram with slope $n^{1/6}$ and take $n \to \infty$, then the process converges to the parabolic Airy line ensemble. Note that by Brownian scaling, this is equivalent to scaling Dyson's Brownian motion around the location $(n, 2n)$ on a $O(n^{2/3})\times O(n^{1/3})$ box, which is the standard KPZ scaling.}%
		\label{fig:melon}%
	\end{figure}
	
	The main obstacle in the paper~\citep{CH} for proving Theorem~\ref{T:airy-BG} was in showing tightness of the line ensembles $\scrC^n$. We give a new proof of this fact in Section~\ref{S:mod-Dyson}, see also Theorem \ref{T:CH-i} in the introduction.
	
	
	We also record here an intuitive lemma from~\citep{CH} which gives monotonicity in the endpoints and boundary conditions for nonintersecting Brownian bridge ensembles. For this lemma, let $\R^k_\ge$ be the set of $\mathbf{x} \in \R^k$ such that $x_i \ge x_{i+1}$ for all $i \in \{1, \dots, k-1\}$.
	
	\begin{lemma}[Lemmas 2.6 and 2.7,~\citep{CH}]
		\label{L:monotone-gibbs}
		Fix $k \in \N$ and $a < b \in \R$.
		Let $\mathbf{w}, \mathbf{x}, \mathbf{y}, \mathbf{z} \in \R^k_\ge$ be such that $w_i \ge x_i$ and $y_i \ge z_i$ for all $i$. Also, let $f, g : [a, b] \to \R \cup \{\pm\infty\}$ be Borel measurable functions such that
		$$
		f(a) \le x_k, \;\;g(a) \le w_k \quad \mathand \quad f(b) \le z_k, \;\; g(b) \le y_k.
		$$
		Finally, assume that $f(r) \le g(r)$ for all $r \in [a, b]$. Then there exists a $2k$-tuple of random functions $(B_1, \dots, B_k, C_1, \dots, C_k)$ where each function has domain $[a, b]$ such that the following holds:
		\begin{enumerate}[nosep, label=(\roman*)]
			\item The sequence $(B_1, \dots, B_k)$ has the distribution of $k$ Brownian bridges with $B_i(a) = x_i$ and $B_i(b) = z_i$, conditioned to avoid $f$ and each other.
			\item  The sequence $(C_1, \dots, C_k)$ has the distribution of $k$ Brownian bridges with $C_i(a) = w_i$ and $C_i(b) = y_i$, conditioned to avoid $g$ and each other.
			\item $C_i(r) \ge B_i(r)$ for all $r \in [a, b]$.
		\end{enumerate}
	\end{lemma}
	
	In this lemma, the definition of nonintersecting bridges starting or ending at the same location should be understood in the same limiting sense as in the definition of Dyson's Brownian motion. Note also that we can consider the case of no upper or lower boundary condition on the bridges (that is, $f$ or $g$ equal to $\pm\infty$). We will also use the limiting case when the endpoints of a few top or bottom bridges are taken to $\pm \infty$, essentially removing them from the conditioning.
	
	\section{Nonintersecting Brownian ensembles} \label{S:nonintersecting}
	
	We first prove a bound on the deviation of the $k$th point in Dyson's Brownian motion.
	
	
	The points at time $1$ in a Dyson's Brownian motion are equal in distribution to the eigenvalues of the Gaussian unitary ensemble. Using this, Ledoux and Rider~\citep{ledoux2010small} show the $k = 1$ case of the following theorem. Their proof extends to general $k$. Note that this theorem can also be deduced from bounds coming from determinantal formulas (i.e. see~\citep{gustavsson2005gaussian}).
	\begin{theorem}
		\label{T:top-bd}
		There exist positive constants $\{c_k, d_k: k \in \N\}$ such that the following holds. Let $W^n_k$ be the $k$th line of an $n$-level Dyson's Brownian motion. Then for all $m \in (0, 5n^{2/3})$ and $n \ge k$ we have
		\begin{align}
		\label{E:maxdk-up}
		\prob(W^n_k(1) - 2 \sqrt{n} \ge m n^{-1/6} ) &\le c_1 e^{-d_1 m^{3/2}}, \quad \mathand \\
		\label{E:maxdk-low}
		\prob(W^n_k(1) - 2 \sqrt{n} \le - m n^{-1/6} ) &\le c_k e^{-d_k m^3}.
		\end{align}
		Also, for all $m \ge 5n^{2/3}$ and $n \ge 1$ we have
		\begin{equation}
		\label{E:large-devW}
		\p(|W^n_k(1) - 2 \sqrt{n}\;| \ge m n^{-1/6}) \le c_1 e^{-d_1 n^{-1/3} m^2}.
		\end{equation}
	\end{theorem}
	
	\begin{proof}
		The $k=1$ case of~\eqref{E:maxdk-up} is shown in~\cite{ledoux2010small} for $m < n^{2/3}$. By replacing the constant $d_1$ with $d_1/5^{3/2}$ we can guarantee that \eqref{E:maxdk-up} holds for all $m < 5n^{2/3}$. Monotonicity of $W^n_k(1)$ in $k$ implies~\eqref{E:maxdk-up} for all $k$.
		
		Next, the bound
		\begin{equation}
		\label{E:p-Wnk-up}
		\p(W^n_k(1) - 2 \sqrt{n} \ge m n^{-1/6}) \le c_1 e^{-d_1 n^{-1/3} m^2}
		\end{equation}
		for $k = 1$ and $m \ge n^{2/3}$ follows from the large deviation theory of the Gaussian unitary ensemble, see~\citep{ledoux2007deviation}, Equation (2.7).  Again, monotonicity of $W^n_k(1)$ in $k$ implies~\eqref{E:p-Wnk-up} for all $k$. Now we have the implications
		$$
		W^n_k(1) - 2 \sqrt{n} \le - m n^{-1/6} \Rightarrow W^n_n(1) - 2 \sqrt{n} \le - m n^{-1/6} \Leftrightarrow -W^n_n(1) - 2 \sqrt{n} \ge  m n^{-1/6}-4\sqrt{n}.
		$$
		Using this, the symmetry $W^n_n \eqd -W^n_1$, and the inequality \eqref{E:p-Wnk-up}, when $m \ge 5 n^{2/3}$ we get the lower tail bound
		\begin{equation}
		\label{E:5n23}
		\p(W^n_k(1) - 2 \sqrt{n} \le - m n^{-1/6}) \le c_1 e^{-d_1 n^{-1/3} m^2/25}
		\end{equation}
		for all $k\le n$. Combining \eqref{E:p-Wnk-up} and \eqref{E:5n23} and replacing the constant $d_1$ by $d_1/25$ yields \eqref{E:large-devW}.
		
		It just remains to prove~\eqref{E:maxdk-low}. The $k=1$ case is shown by Ledoux and Rider~\cite{ledoux2010small}. For the general $k$ case, we use the following simple extension of their argument.
		Let $H^c$ be the $n\times n$ symmetric tridiagonal matrix defined by
		$$
		H^c_{i,j}=\begin{cases}
		-2c\sqrt{n}-\frac{ci}{\sqrt{n}}+N_i/\sqrt{2}, & i=j
		\\ (\chi_{2(n-i \wedge j)} - \expt \chi_{2(n-i \wedge j)})/\sqrt{2}+c\sqrt{n} &|i-j|=1
		\end{cases}
		$$
		where the $N_k$ are independent standard normal random variables, and the $\chi_{2(n-j)}$ are independent $\chi$ random variables with parameter $2(n-j)$. Lemma 5 in~\cite{ledoux2010small} shows that for a universal constant $c$, almost surely
		$$
		\langle vH^c,v\rangle \le \langle v(T - 2\sqrt{n} I_n),v\rangle
		$$
		for all $v \in \R^n$, where $T$ is a tridiagonal matrix whose eigenvalues are $W^n_k(1)$. Note that in the above equation and throughout the proof, we use the notation $vH := v^T H$ to denote left multiplication by a row vector.
		
		Therefore by~\cite{horn2012matrix}, Corollary 4.3.3, letting $\la^k(H^c)$ be the $k$th largest eigenvalue of $H^c$, we have $\lambda_k(H^c)\le W^n_k(1)-2\sqrt{n}$.
		
		
		Let $v_1,\ldots, v_k$ be unit vectors whose supports as subsets of $\{1,\ldots, n\}$  have distance at least $3$ from each other. Let $P$ be the orthogonal projection onto $V = \text{span}\{v_1, \dots, v_k\}$. We consider $H_k=PH^cP$ as a linear operator on the $k$-dimensional subspace $V$. Since $H^c$ is tridiagonal and the supports of the $v_i$ are separated by distance $2$, $PH^cP$ is diagonal in the basis $\{v_i\}$ of $V$.
		Thus we have
		\begin{equation}
		\label{E:Wnk-4}
		W_k^n(1)-2\sqrt{n}\ge\lambda_k(H^c)\ge \lambda_k(H_k)= \min_{i=1}^k \langle v_i H_k ,v_i\rangle.
		\end{equation}
		The second inequality follows since orthogonal projections do not increase $k$th top eigenvalues, see~\cite{horn2012matrix}, Corollary 4.3.16. It remains to find suitable vectors $v_1, \dots, v_k$ that make the right hand side above large.
		
		
		Let $b_k \in (0, 1)$ be a small constant to be specified later. Let $\ell=\ceil{b_k mn^{1/3}}$ and let
		$$v^n_i(j+3(i-1)\ell) = j/\ell\wedge \left(1-j/\ell\right) \,\,\,\mbox{ for }\,\, 1\le j \le 2\ell$$
		and set $v^n_i(j) = 0$ otherwise.

		First assume that $m$ satisfies $m b_k \in [2, n^{2/3}/(3k)]$. In this case, the supports of $v^n_i$ are distance at least $3$ from each other and all the vectors $v^n_i$ are well-defined in $\R^n$. To prove~\eqref{E:maxdk-low} for $m$ in this range it is enough to show that
		\begin{equation}\label{E:LR-bound}
		\p(\langle v^n_iH_k,v^n_i\rangle < - mn^{-1/6}\|v^n_i\|^2)<c_ke^{-d_km^3}.
		\end{equation}
		The claim~\eqref{E:maxdk-low} follows from~\eqref{E:LR-bound} by a union bound and the bound~\eqref{E:Wnk-4} applied to the normalized vectors $v^n_k/\|v^n_k\|$. The proof of \eqref{E:LR-bound} for $v^n_1$ is on page 1331 of~\cite{ledoux2010small}. We use the same argument for the $k > 1$ case.
		
		
		First, since $P v^n_i = v^n_i$ for all $i$, we have that
		$$
		\langle v^n_iH_k,v^n_i\rangle = \langle v^n_iP H^c, Pv^n_i\rangle = \langle v^n_iH^c, v^n_i\rangle= H_c(v^n_i, N) + \chi(v^n_i),
		$$
		where $N = (N_1, \dots, N_n)$ is a vector of $n$ independent standard normal random variables,
		\begin{align*}
		H_c(v, g) &= \frac{1}{\sqrt{2}} \sum_{j=1}^n g_j v(j)^2 - c\sqrt{n} \sum_{j=0}^n (v(j + 1) - v(j))^2 - \frac{c}{\sqrt{n}} \sum_{j=1}^n j [v(j)]^2\quad \mathand \\
		\chi(v) &= \sqrt{2} \sum_{j=1}^{n-1} [\chi_{2(n - j)} - \expt \chi_{2(n - j)}] v(j)  v(j+1).
		\end{align*}
		Here it is understood that $v(0) = v(n+1) = 0$. By a union bound, the left side of \eqref{E:LR-bound} is bounded above by
		\begin{equation}
		\label{E:Hcc}
		\p\lf(H_c(v^n_i, N) < - \frac{m}2 n^{-1/6} ||v_i^n||^2\rg) + \p\lf(\chi(v^n_i) < - \frac{m}2 n^{-1/6} ||v_i^n||^2\rg).
		\end{equation}
		To bound these probabilities, we will use that
		\begin{equation}
		\label{E:asymptotics'}
		\begin{split}
		&||v^n_i||^2 \sim \sum_{j=1}^n [v^n_i(j)]^4 \sim \sum_{j=1}^{n-1} [v^n_i(j)]^2 [v^n_i(j+1)]^2 \sim b_k m n^{1/3}, \\
		&\sum_{j=0}^n (v^n_i(j + 1) - v^n_i(j))^2 \sim \frac{1}{b_k mn^{1/3}}, \quad \mathand \quad \sum_{j=1}^n j [v^n_i(j)]^2 \sim k b_k^2 m^2 n^{2/3},
		\end{split}
		\end{equation}
		where the notation $a \sim b$ here means that the ratio $a/b$ is bounded above and below by positive constants, uniformly over all choices of $n \in \N$ and $m$ satisfying $m b_k \in [2, n^{2/3}/(3k)]$. We start with the first term in \eqref{E:Hcc}. As long as the constant $b_k$ was chosen small enough (i.e. taking $b_k = \ep/k$ for a small $\ep > 0$ works), the asymptotics in \eqref{E:asymptotics'} guarantee that
		\begin{align}
		\nonumber
		\p\lf(H_c(v^n_i, N) < - \frac{m}2 n^{-1/6} ||v_i^n||^2\rg) &\le \p\lf(\sum_{j=1}^n N_j v(j)^2 \le - \frac{m}4 n^{-1/6} ||v_i^n||^2  \rg).
		\nonumber
		\end{align}
		Converting the sum of independent normals to a single normal $N_1$, the right hand side above is equal to
		\begin{align}
		\label{E:N1-bound}
		\p\lf(N_1 \sqrt{\sum_{j=1}^n v(j)^4} \le - \frac{m}4 n^{-1/6} ||v_i^n||^2  \rg).
		\end{align}
		By the asymptotics in \eqref{E:asymptotics'}, this is bounded above by
		\begin{equation}
		\label{E:plff}
		\p\lf(N_1 \le - c b^{1/2}_k m^{3/2}  \rg)\le e^{-d_k m^3},
		\end{equation}
		where $d_k$ depends on $k$ only.
		We now turn to the second probability in \eqref{E:Hcc}. By Lemma 7 in \cite{ledoux2010small}, for any chi distributed random variable $\chi$ with parameter greater than or equal to $1$ and $\la > 0$, we have that
		\begin{equation}
		\label{E:lachi}
		\expt e^{-\la (\chi - \expt \chi)} \le e^{\la^2/2}.
		\end{equation}
		By Markov's inequality and \eqref{E:lachi}, for $r, \la > 0$ and any vector $v$ we have
		\begin{align*}
		\p\lf( \chi(v) \le -r ||v||^2 \rg) \le e^{-\la r ||v||^2} \expt e^{-\la \chi(v)} \le \exp \lf(-\la r ||v||^2 + \frac{\la^2}{\sqrt 2} \sum_{k=1}^{n-1} v(j)^2 v(j+1)^2 \rg).
		\end{align*}
		Optimizing over $\la$ and using the asymptotics in \eqref{E:asymptotics'} yields
		\begin{align*}
		\p\lf( \chi(v) \le -r ||v||^2 \rg) \le \exp \lf( - r^2 ||v||^4 \big\slash \lf( 4\sum_{k=1}^{n-1} v(j)^2 v(j+1)^2 \rg) \rg) \le e^{-d_k m^3},
		\end{align*}
		completing the proof of \eqref{E:LR-bound}.
		
		
		To extend the bound \eqref{E:maxdk-low} from $m \in [2/b_k, n^{2/3}/(3kb_k)]$ to all $m \in (0, 5n^{2/3}]$, first
		note that we can extend the upper bound on $m$ by a constant factor by replacing the constant $d_k$ by a smaller one. Finally, we can extend the lower bound to $0$ by making $c_k$ large enough so that the claim trivially holds for $m \le 2/b_k$.
	\end{proof}
	
	
	Our first use of this theorem is a modulus of continuity for Brownian bridges conditioned not to intersect.
	
	\begin{lemma}
		\label{L:bridge-disc} There exist constants $c,d$ so that the following holds. Let $k\in \N$, $B_i$, $i\in \{1,\ldots k\}$ be Brownian bridges
		with arbitrary start points $B_1(a) \ge \dots \ge B_k(a)$ and end points $B_1(b) \ge \dots \ge B_k(b)$, and slopes
		$$
		b_i = \frac{B_i(b) - B_i(a)}{b - a}
		$$
		on some time interval $[a,b]$, conditioned not to intersect. Then for any $i$, $t>0$, $s,s+t \in [a,b]$  we
		have
		$$
		\p(|B_i(s+t)-B_i(s)-b_it|\ge m\sqrt{t})\le e^{ck-dm^2}
		$$
		for universal constants $c, d> 0$.
	\end{lemma}
	
	Lemma~\ref{L:bridge-disc} is a more precisely stated version of Proposition~\ref{P:nonint-dev-i}.
	
	\begin{proof}
		By Brownian scaling, we may assume that $[a,b]=[0,1]$. We can further assume that $B_i(0)=B_i(1)=0$ for the $i$ in question. Also, by the time and value-reversal symmetry of the problem, we may assume $s\le 1/2$. For the price of a factor of $2$ we can just prove the upper bound on $B_i(s+t)-B_i(s)$.
		
		
		For this, consider the conditional processes $B_j:[s, 1] \to \R$ conditioned on the values of all the $B_j$ at time $s$. By the Brownian Gibbs property, the conditional law of these processes is that of $k$ independent Brownian bridges on $[s, 1]$ conditioned not to intersect. By Lemma~\ref{L:monotone-gibbs}, the conditional law of $B_i|_{[s, 1]}$ is dominated by moving the starting and ending points (at times $s$ and $1$) of all the other Brownian bridges up while keeping their order. Those above $B_i$ are moved to $(\infty,\infty)$ and those below are moved to $(B_i(s),B_i(1))$. Therefore $B_i|_{[s, 1]}$ can be coupled monotonically with a linear function from $B_i(s)$ to $B_i(1)$ plus the top line $C^{k-i+1}_1$ of an independent Brownian $(k-i+1)$-melon on the interval $[s, 1]$. This gives
		\begin{equation}
		\label{E:BiX}
		\begin{split}
		B_i(s+t)-B_i(s)\le \frac{-B_i(s)t}{1-s}+C^{k-i+1}_1(s+t).
		\end{split}
		\end{equation}
		By a similar domination argument with Lemma~\ref{L:monotone-gibbs}, the function $B_i:[0, 1] \to \R$ can be coupled with the bottom line $D^i_i$ of a Brownian $i$-melon with domain $[0, 1]$ so that $B_i \ge D^i_i$. Therefore the right hand side of~\eqref{E:BiX} is bounded above by
		$$
		\frac{-D^i_i(s)t}{1-s}+C^{k-i+1}_1(s+t).
		$$
		By the relationship~\eqref{E:melon-dyson} between Dyson's Brownian motion and Brownian $k$-melons, and by Brownian scaling, we have
		\begin{equation}
		\label{E:DiY}
		\frac{-D^i_i(s)t}{1-s}+C^{k-i+1}_1(s+t) \eqd \frac{t\sqrt{s(1-s)}}{1-s}Y^i_1(1)+\sqrt\frac{t(1-s-t)}{1-s}X_1^{k-i+1}(1),
		\end{equation}
		where $Y^i$ and $X^{k-i+1}$ are independent Dyson's Brownian motions with $i$ and $k - i + 1$ lines, respectively. Since $Y^i_1$ and $X^{k-i+1}_1$ are bounded above by their positive parts, we can bound the right hand side of~\eqref{E:DiY} above by
		\begin{equation}
		\label{E:frac-tt}
		\sqrt{t}\lf(\sqrt{\frac{t s}{1 -s}}\max(Y^i_1(1),0) + \sqrt{\frac{1 - s - t}{1-s}}\max(X_1^{k-i+1}(1), 0)\rg).
		\end{equation}
		Since $t \in [0, 1-s]$ and $s \in [0, 1/2]$, both of the square root factors inside the brackets in \eqref{E:frac-tt} are bounded above by $1$. Hence \eqref{E:frac-tt} is bounded above by
		\begin{equation}
		\sqrt{t}\lf(\max(Y^i_1(1),0) + \max(X_1^{k-i+1}(1), 0)\rg).
		\end{equation}
		Now, by ~\eqref{E:large-devW} in Theorem~\ref{T:top-bd}, for any $m \ge 7\sqrt{k}$ we have that
		\begin{align}
		\nonumber
		\p(Y^i_1(1) \ge m) &= \p\lf(Y^i_1(1) - 2\sqrt{i} \ge (m - 2\sqrt{i}) k^{1/6} k^{-1/6} \rg) \\
		\nonumber
		&\le c_1 \exp \lf(-d_1 (m - 2\sqrt{i})^2 \rg) \\
		\label{E:Y-bdd}
		& \le c_1 \exp \lf(-\frac{d_1}{2} m^2 \rg).
		\end{align}
		Here $c_1, d_1$ are independent of $i, k$. To apply the bound ~\eqref{E:large-devW} in Theorem~\ref{T:top-bd}, we have used that $(m - 2\sqrt{i})k^{1/6} \ge 5 k^{2/3}$ whenever $m \ge 7\sqrt{k}$. Now, we can choose $c$ large enough, indepedent of $k$, so that $\frac{1}2 e^{ck} \ge c_1$ and
		\begin{equation}
		\label{E:123}
		\frac{1}{2} \exp \lf(ck -\frac{d_1}{2} m^2 \rg) \ge 1
		\end{equation}
		for $m \le 7 \sqrt{k}$, for all $k \in \N$. Therefore by combining \eqref{E:Y-bdd} and \eqref{E:123}, we have that
		\begin{equation}
		\p(Y^i_1(1) \ge m) \le \frac{1}{2} \exp \lf(ck -\frac{d_1}{2} m^2 \rg)
		\end{equation}
		for all $m$. The same bound holds for $X^{k-i+1}_1$, so by a union bound we have that
		$$
		\p\lf(\sqrt{t}\lf(\max(Y^i_1(1),0) + \max(X_1^{k-i+1}(1), 0)\rg) \ge m \sqrt{t} \rg) \le \exp \lf(ck -\frac{d_1}{2} m^2 \rg),
		$$
		as desired.
	\end{proof}
	
	Modulus of continuity results naturally follow from statements such as the one in Lemma~\ref{L:bridge-disc}. The classical example of this is L\'evy's modulus of continuity of Brownian motion. 
	For future use (most notably in the paper~\cite{DOV}), we state this next lemma in greater generality than we need it here.
	
	
	\begin{lemma}
		Let $T=I_1\times \dots \times I_d$ be a product of bounded real intervals of length $b_1, \dots, b_d$. Let $c, a>0$.
		Let $\scrH$ be a random continuous function from $T$ taking values in a real vector space $V$ with Euclidean norm $|\cdot|$. Assume that for every $i \in \{1, \dots, d\}$, that there exist $\al_i \in (0,1), \beta_i, r_i > 0$ such that
		\label{L:levy-est}
		\begin{equation}
		\label{E:tail-bd}
		\p(|\scrH(t+e_i u) - \scrH(t)| \ge m u^{{\alpha_i}}) \le c e^{-a{m^{{\beta_i}}}}
		\end{equation}
		for every coordinate vector $e_i$, every $m>0$, and every $t,t+u e_i\in T$ with $u < r_i$. Set $\beta = \min_i \beta_i, \al = \max_i \al_i$, and $r = \max_i r_i^{\al_i}$. Then with probability one we have
		\begin{equation}
		|\scrH(t + s) - \scrH(t)| \le C \lf(\sum_{i=1}^d |s_i|^{\al_i} \log^{1/\beta_i} \lf(\frac{2 r^{1/\al_i}}{|s_i|} \rg) \rg),
		\end{equation}
		for every $t,t+s\in T$ with $|s_i| \le r_i$ for all $i$ (here $s = (s_1, \dots, s_d)$).
		Here $C$ is random constant satisfying
		$$
		\p(C > m) \le \lf[\prod_{i=1}^d \frac{b_i}{r_i} \rg] c c_0 e^{-c_1 m^{\beta}},
		$$
		where $c_0$ and $c_1$ are constants that depend on $\al_1, \dots, \al_d, \beta_1, \dots, \beta_d, d,$ and $a$. Notably, they do not depend on $b_1, \dots, b_d ,c$ or $r_1, \dots, r_d$.
	\end{lemma}
	
	Note that the above lemma can also be extended to the case when $\al_i = 1$, but the power of $1/\beta_i$ in the logarithm term changes to a power of $1 + 1/\beta_i$.
	
	
	The proof of Lemma~\ref{L:levy-est} mimics the proof of Levy's modulus of continuity of Brownian motion. The piecewise linear approximations of Brownian motion used in that proof must be replaced by polynomial approximations in each variable. To set up these approximations, we will need an auxiliary lemma.
	
	\begin{lemma}
		\label{L:poly-prop}
		Let $d \in \N$, and let $B =\prod_{i=1}^d [b_i, c_i]$  be any box in $\R^d$ with $b_i < c_i$ for all $i$. Let $a_v$ be real numbers indexed by $v \in \prod_{i=1}^d \{b_i, c_i\}$. There exists a unique polynomial $h:B \to \R$ which is linear in each variable such that $h(v) = a_v$ for all $v \in \prod_{i=1}^d \{b_i, c_i\}$.
	\end{lemma}
	
	\begin{proof}
		Without loss of generality, we may shift the box so that $b_i = 0$ for all $i$.
		For $w \sset \{1, \dots, d\}$, let $x_w = \Pi_{i \in w} x_i$, and define $h(x) = \sum m_w x_w$, where the sum ranges over $w \sset \{1, \dots, d\}$.
		We want to choose the coordinates $m_w$ so that $h(v) = a_v$ for all $v$. This gives a linear system of $2^d$ equations with $2^d$ unknowns $m_w$. Let $A$ be the coefficient matrix of this linear system with rows indexed by $v$'s and columns indexed by $w$'s. An entry $A_{v, w}$ is nonzero exactly when
		$$
		w \sset \{i \in \{ 1, \dots, n\} : v_i = c_i \}.
		$$
		Denoting the right hand side above by $w(v)$, we therefore have $\det A = \pm \prod_v A_{v, w(v)},$ as no other permutations contribute to the determinant. Hence $A$ is invertible, so the system of equations has a unique solution, yielding $h$.
	\end{proof}

	\begin{proof}[Proof of Lemma~\ref{L:levy-est}]
		We first consider the case when $T = \prod_{i=1}^d [0,1]$ and $r_i = 1$ for all $i$.
		
		
		We will prove the bound when $s$ is a multiple of the unit vector $e_i$ for some $i \in \{1, \dots, d\}$. The result then follows by the triangle inequality. For ease of notation, we set $i = 1$.  Let $k \in \N$ be large enough so that $k\al_j  > \al_1$ for all $j \in \{2, \dots, d\}$.
		For $n \in \N$, let
		$$
		P_n^1 =  \{0, 1/2^n, 2/2^n, 3/2^n,\ldots, 1\} \quad \mathand \quad P_n^j =  \{0, 1/2^{nk}, 2/2^{nk}, \ldots, 1\}
		$$
		for $j \in \{2, \dots, d\}$, and set $P_n = \prod_{i=1}^d P_n^i$. Define the set of translated boxes
		$$
		C_n = \lf\{S = v + [0,1/2^n] \X \prod_{j=2}^d [0, 1/2^{nk}] : v \in P_n, S \sset T \rg\}.
		$$
		For $n \in \{0, 1, \dots\}$, define approximations $\scrH_n$ of $\scrH$, by setting $\scrH_n=\scrH$ on $P_n$, and by requiring that $\scrH_n$ is linear in each variable on every box $D\in C_n$. For a fixed box $D$, such an approximation exists and is unique by Lemma~\ref{L:poly-prop}. Gluing these approximations together yields $\scrH$ (note that the approximations agree where the boxes overlap by the uniqueness in Lemma~\ref{L:poly-prop}). Observe that $\scrH_n\to \scrH$ uniformly on $T$. We also set $\scrH_{-1}$ to be constantly equal to $\scrH(0, \dots, 0)$.
		
		
		Now let $t, t +s \in T$ be vectors that differ only on the first coordinate. Setting $\scrG_n = \scrH_n - \scrH_{n-1}$, we can write $\scrH(t +s) - \scrH(t) = \sum_{n=0}^\infty \scrG_n(t+s) - \scrG_n(t)$. Letting $||\cdot||$ denote the uniform norm on functions, by linearity of $\scrH_n$ in the first coordinate on every box $S\in C_n$ we have $\|\partial_1 \scrG_n\| \le 2^n \|\scrG_n\|$. Therefore since $s$ only has a nonzero first coordinate, by the mean value theorem, we have
		\begin{equation}
		\label{E:G-deriv}
		|\scrG_n(t + s) - \scrG_n(t)| \le |s| 2^n \|\scrG_n\|.
		\end{equation}
		Additionally, we have the bound
		\begin{equation}
		\label{E:G-tri}
		|\scrG_n(t + s) - \scrG_n(t)| \le |\scrG_n(t + s)| + |\scrG_n(t)| \le 2\|\scrG_n\|.
		\end{equation}
		Hence to estimate $\scrH(t+s) - \scrH(t)$, we need a good bound on $\|\scrG_n\|$.
		
		
		For $n \ge 1$ and a given box $S \in C_{n-1}$, the values of $\scrH_{n-1}$ on $S$ are convex combinations of the values of $\scrH$ on the vertices of $S$. Also, the values of
		$\scrH_{n}$ on $S$ are convex combinations of the values of $\scrH$ on $S\cap P_n$ for any $n \ge 1$. In particular, this implies that for $n \ge 1$,
		\begin{align*}
		\max_{t \in S} |\scrH_{n-1}(t)-\scrH_{n}(t)| &\le \max_{t + s, t\in S\cap P_n} |\scrH(t + s)-\scrH(t)| \\
		&\le \max_{t + s, t\in P_n, |s_1| \le 2^{-(n -1)}, |s_j| \le 2^{-k(n - 1)}} |\scrH(t + s)-\scrH(t)|.
		\end{align*}
		Here the final inequality follows since $S$ is a translate of the box $[0,1/2^{n-1}] \X \prod_{j=2}^d [0, 1/2^{(n-1)k}]$.
		Therefore for $n \ge 1$,
		\begin{equation}
		\label{E:Gng}
		||\scrG_n|| \le \max \{|\scrH(t + s) - \scrH(t)| : t, t + s \in P_n, |s_1| \le 2^{-(n -1)}, |s_j| \le 2^{-k(n - 1)}\}.
		\end{equation}
		Because of how we defined $\scrH_{-1}$, the above bound also holds for $\scrG_0$. Now, for any $t, t + s \in P_n$ with $|s_1| \le 2^{-(n -1)}, |s_j| \le 2^{-k(n - 1)}$, the lattice structure of $P_n$ implies that
		we can write 
		$$
		s = a_1 2^{-n} e_1 + \dots + a_k 2^{-nk} e_k
		$$
		where the $e_i$ are coordinate vectors, $a_1 \in \{0, 1, 2\}$ and $a_i \in \{0, 1, \dots, 2^k\}$ for $i > 1$. Moreover, for any 
		$$
		s' = b_1 2^{-n} e_1 + \dots + b_k 2^{-nk} e_k
		$$
		with $b_i \in \{0, \dots, a_i\}$ for all $i$, we have $t + s' \in P_n$ as well. Therefore by \eqref{E:Gng} and the triangle inequality, for $n \ge 0$ we have
		\begin{equation*}
		\|\scrG_n\|\le N_{n,1}+ \cdots +N_{n,d},
		\end{equation*}
		where
		\begin{align*}
		N_{n,1} &= 2\max \{|\scrH(t + s)-\scrH(t)| : t + s, t \in P_n, |s_1| = 2^{-n}, s_j = 0, j \ge 2 \} \quad \mathand \\
		N_{n,j}&=2^{k} \max \{|\scrH(t+s)-\scrH(t)| : t + s, t \in P_n, |s_j| = 2^{-nk}, s_i = 0, i \ne j \}.
		\end{align*}
		By~\eqref{E:tail-bd} and a union bound over $|C_n|$ boxes, for all $m > 0$ we have that
		\begin{align}
		\p\lf(N_{n,1} > 2^{1-n\al_1} m((n+1)kd)^{1/\beta_1}\rg) &\le ce^{(n + 1)kd (\log 2 - a m^{\beta_1})} \quad \mathand \\
		\p\lf(N_{n,i} > 2^{k-nk\al_i} m((n+1)kd)^{1/\beta_i}\rg) &\le ce^{(n + 1)kd (\log 2 - a m^{\beta_i})}
		\end{align}
		for $i \in \{2, \dots, d\}$. Here we have used that $|C_n| \le 2^{(n+1)kd}$.
		Set
		$$
		D= \sup_{n \ge 0} \frac{||\scrG_n|| 2^{n\al_1}}{(n+1)^{1/\beta_1}}.
		$$
		A tail bound on $D$ can be obtained via the above bounds on $N_{n, j}$. Using that $k\al_j > \al_1$ for $j \ne 1$, we can conclude that
		$$
		\p(D \ge m) \le c c_0 e^{-c_1m^{\beta}}
		$$
		where $\beta = \min_i \beta_i$, and $c_0$ and $c_1$ are constants that depends on the terms $\al_i, \beta_i, a,$ and $d$.
		Now by using the bounds in~\eqref{E:G-deriv} and~\eqref{E:G-tri}, for every $t + s, t \in T$ which differ only on the first coordinate, we have
		\begin{align*}
		|\scrH(t + s) - \scrH(t)| &\le \sum_{n = 0}^\infty D (n+1)^{1/\beta_1} 2^{-n \al_1} (|s|2^n\wedge 2) \\
		&\le c_2 D |s|^{\al_1} \log^{1/\beta_1} \lf(\frac{2}{|s|} \rg)
		\end{align*}
		for a constant $c_2$ that depends on the terms $\al_1$ and $\beta_1$. This completes the proof in the case when $T = [0, 1]^d$ and $r_i = 1$ for all $i$.
		
		
		For general lengths $b_i$ and $r_i = b_i$, by translation we can assume that $T = \prod_{i=1}^d [0, b_i]$. Let $\al = \max_i \al_i$ and $b = \max_i b_i^{\alpha_i}$, and define the process
		$$
		U(t_1, \dots, t_d) = b^{-1} \mathcal H(b^{1/\alpha_1} t_1, b^{1/\alpha_2} t_2 \dots, b^{1/\alpha_d} t_d).
		$$
		The process $U$ satisfies the same tail bound~\eqref{E:tail-bd} as $\scrH$ on the subset $S \sset [0, 1]^d$ where it is defined. We can extend $U$ to a function satisfying~\eqref{E:tail-bd} on all of $[0, 1]^d$ by letting $P:[0, 1]^d \to S$ be the projection map sending $x$ to the closest point to $x$ in $S$, and letting $U(x) = U(P(x))$. This extension of $U$ satisfies the conclusion of the lemma by the $r_i = b_i = 1$ case, giving the desired bound on $\scrH$ on $T$.
		
		For the most general case when $r_i < b_i$ for some $i$, we can get individual modulus of continuity bounds on boxes with side lengths bounded above by $r_i$ and whose union is $T$. We require at most $2^d \prod b_i/r_i$ many boxes. A union bounds then yields the lemma at the expense of the $\prod b_i/r_i$ factor in the bound on $C$.
	\end{proof}
	Lemmas~\ref{L:bridge-disc} and~\ref{L:levy-est} now immediately imply the following.
	\begin{prop}
		\label{P:bridge-bd}
		There exist universal constants $c,d > 0$ so that the following holds. Let $k\in \N$ and let $B_i$, $i\in\{1,\ldots k\}$ be independent Brownian bridges with arbitrary start and end points and slope $b_i$ on some time interval $[t_1, t_2]$, conditioned not to intersect. There are random constants $C_i$ satisfying
		$$
		\p(C_i > m) \le e^{ck -dm^2}
		$$
		for any $m > 0$, so that for any $i$, $t>0$, $s,s+t \in [t_1, t_2]$  we have
		$$
		|B_i(s+t)-B_i(s)-b_it|\le C_i \sqrt{t}\log^{1/2}\lf(\frac{2\;|t_2 - t_1|}t\rg).
		$$
	\end{prop}
	
	\begin{proof}
		Fix $i$, and consider the process
		$$
		\tilde B_i(t) = B_i(a + t) -B_i(a) - b_i t
		$$
		By Lemma~\ref{L:bridge-disc}, the process $\tilde B_i$ satisfies
		the assumptions of Lemma~\ref{L:levy-est} in dimension $1$ with  $\al_1 = 1/2, \beta_1 = 2, b_1 = r_1 = t_2 - t_1.$ The only dependence on $k$ in Lemma~\ref{L:bridge-disc} is in the constant $e^{ck}$. This translates to an $e^{ck}$ factor in the tail bound on $C_i$ (for a possibly different $c$) and has no other effect.
	\end{proof}
	
	\section{Modulus of continuity for top Dyson lines}\label{S:mod-Dyson}
	
	Before proceeding with the main goal of the paper, which is to understand the Airy line ensemble, we give a short proof of tightness of the Dyson lines in the scaling~\eqref{E:rescaled-dyson}. The ideas in this proof will be used later in the paper to give a similar modulus of continuity result for parabolic Airy lines. As discussed in the introduction, Proposition~\ref{P:dyson-tails} is a restatement of Proposition~\ref{P:Dyson-tails-i} and was previously obtained in Hammond \cite{hammond2016brownian}, Theorem 2.14.
	
	\begin{prop}
		\label{P:dyson-tails}
		Fix $k \in \N$ and $c > 0$. There exist constants $c_k, d_k > 0$ such that for every $n \in \N$, $t > 0, \;s \in (0, ctn^{-1/3}]$, and $m > 0$ we have
		$$
		\p\bigg(\lf| W_k^n(t) - W_k^n(t +s) - \frac{s\sqrt{n}}{\sqrt{t}}\rg| > m \sqrt{s} \bigg) \le c_ke^{-d_km^{3/2}}.
		$$
	\end{prop}
	
	
	Throughout the proof, $c_k$ and $d_k$ will be constants that depend only on $k$ and $c$, but may change from line to line.
	\begin{proof} We can assume that $m > 4c^{3/2}$ by possibly changing $c_k$ and $d_k$.
		By Brownian scaling, it suffices to prove the lemma for $t = 1$. Now, by time inversion and Brownian scaling, for any times $p, q \in \R$ we have that
		\begin{equation}
		\label{E:Wnk-p}
		W^n_k(p) - W^n_k(q) \eqd \sqrt{\frac{p}{q}} W^n_k(q) - \sqrt{\frac{q}{p}} W^n_k(p).
		\end{equation}
		To apply this property, we first define the error $C^n_k(s)$ by
		$$
		W^n_k(1 + s) = \lf(2 + s - \frac{s^2}4 \rg)\sqrt{n} + n^{-1/6}C^n_k(s).
		$$
		Note that $(2 + s - s^2/4)\sqrt{n}$ is the second order Taylor expansion of $2\sqrt{1 + s}\sqrt{n}$, so for small $s$ it is a good approximation of $W^n_k(1 + s)$. In particular, by Theorem~\ref{T:top-bd} and Brownian scaling, the error term $C^n_k$ satisfies the tail bound
		\begin{equation}
		\label{E:Cnk-tail}
		\p(|C^n_k(s)| \ge r) \le c_k e^{-d_k r^{3/2}}
		\end{equation}
		for all $s \in [0, ctn^{-1/3}]$. Using the property~\eqref{E:Wnk-p} applied to the difference
		$
		W^n_k(1) - W^n_k(1 + s)
		$
		and Taylor expanding the resulting $\sqrt{1 + s}$ and $1/\sqrt{1 + s}$ terms gives that
		\begin{equation}
		\label{E:Wnk-phi}
		W^n_k(1) - W^n_k(1 + s) + s\sqrt{n} \eqd W^n_k(1 + s) - s\sqrt{n} - W^n_k(1) - \phi_n(s),
		\end{equation}
		where $\phi_n(s)$ is a random error term that satisfies
		\begin{equation}
		\label{E:phii}
		|\phi_n(s)| \le d\sqrt{n}|s|(s^2 + n^{-2/3}|C_k^n(0)| + n^{-2/3}|C_k^n(s)|)
		\end{equation}
		for a constant $d$ that depends on the width $c$ of the interval but not on $n$, or the point $s$. By using~\eqref{E:Wnk-phi}, we can write
		\begin{align*}
		\p\bigg(&\lf| W^n_k(1) - W^n_k(1 + s) + s\sqrt{n}\rg| > m \sqrt{s} \bigg)
		\\
		&\le 2\p\lf(W^n_k(1) - W^n_k(1 + s) + s\sqrt{n} > \frac{m}2 \sqrt{s} \rg) + \p\lf(\phi_n(s) > \frac{m}2 \sqrt{s} \rg).
		\end{align*}
		By the bound on $\phi_n$, and the bound in~\eqref{E:Cnk-tail}, we have that
		\begin{align*}
		\p\lf(\phi_n(s) > \frac{m}2 \sqrt{s} \rg)
		\le c_k e^{-d_km^{3/2}}.
		\end{align*}
		It remains to bound
		\begin{equation}
		\label{E:P-cc}
		\p\lf(W^n_k(1) - W^n_k(1 + s) + s\sqrt{n} > \frac{m}2 \sqrt{s} \rg).
		\end{equation}
		The method we use here will be applied again in Lemma~\ref{L:airy-tails}. Set
		$$
		r = \frac{m}{2\sqrt{sn}},
		$$
		and let $L$ be the line with $L(1) = W_k^n(1)$ and $L(1 + r) = W_k^n(1 + r)$. Note that our assumption that $m > 4c^{3/2}$ implies that $r > s$. The Brownian Gibbs property for Dyson's Brownian motion and monotonicity (Lemma~\ref{L:monotone-gibbs}) implies that in the interval $[1, 1 + r]$, the line $W_k^n$ stochastically dominates $L + B$, where $B: [1, 1+ r] \to \R$ is the bottom line of a Brownian $k$-melon on $[1, 1 + r]$ that is independent of $W^n_k(1)$ and $W^n_k(1 +r)$ (this is what we get after moving the bottom boundary to $-\infty$ and the top $k-1$ boundary points to $L(1)$ and $L(1+r)$). Hence~\eqref{E:P-cc} is bounded above by
		\begin{align}
		\nonumber
		\p&\lf([L+B](1) - [L + B](1 + s) + s\sqrt{n} > \frac{m\sqrt{s}}2  \rg) \\
		\label{E:L-chhk}
		&\le \p\lf(L(1) - L(1 + s) + s\sqrt{n}> \frac{m\sqrt{s}}{4} \rg) + \p\lf(B(1 + s) < -\frac{m\sqrt{s}}4 \rg).
		\end{align}
		By Lemma~\ref{L:bridge-disc}, the second term in \eqref{E:L-chhk} is bounded above by $c_ke^{-d_km^2}$ for constants $c_k$ and $d_k$. For the first term, since $L$ is linear and equal to $W^n_k$ at $1$ and $1 + r$ we can write
		$$
		L(1) - L(1 + s) + s \sqrt{n} = \frac{s}{r} \lf(W^n_k(1) - W^n_k(1+ r) + r\sqrt{n}\rg).
		$$
		Therefore the first term in \eqref{E:L-chhk} is equal to
		$$
		\p\lf(W_k^n(1) - W_k^n(1 + r) + r\sqrt{n}> \frac{mr}{4\sqrt{s}} \rg).
		$$
		By a union bound, this is bounded above by
		\begin{align}
		\nonumber
		&\p\lf(W_k^n(1) > 2\sqrt{n} + \frac{mr}{16\sqrt{s}} \rg) + \p\lf(-W_k^n(1 + r) + r\sqrt{n} > -2\sqrt{n} + \frac{3mr}{16\sqrt{s}}\rg) \\ = \; &\p\lf(W_k^n(1) > 2\sqrt{n} + \frac{mr}{16\sqrt{s}} \rg) + \p\lf(W_k^n(1 + r) <\lf(2 + r -\frac{r^2}4\rg)\sqrt{n} - \frac{mr}{16\sqrt{s}} \rg).
		\label{E:splat}\end{align}
		In the second above we have used that $\sqrt{n}r^2/4 = mr/(8\sqrt{s})$. Now, since $s$ is bounded above by $cn^{-1/3}$, we have that
		$$
		\frac{m r}{16\sqrt{s}} \ge \frac{m^2}{32c} n^{-1/6}.
		$$
		Hence the first probability in~\eqref{E:splat} is bounded above by $c_k e^{-d_km^3}$ by the tail bounds on Dyson's Brownian motion established in Theorem~\ref{T:top-bd}. By Brownian scaling, the second term in \eqref{E:splat} is equal to
		\begin{equation*}
		\p\lf(\sqrt{1 + r} W_k^n(1) < \lf(2 + r -\frac{r^2}4\rg)\sqrt{n} - \frac{mr}{16\sqrt{s}} \rg).
		\end{equation*}
		Since $2\sqrt{1 + r} \ge 2 + r - r^2/4,$ this is bounded above by
		\begin{equation}
		\label{E:p-tobd}
		\p\lf(W_k^n(1)< 2\sqrt{n} - mn^{-1/6} \frac{r n^{1/6}}{16\sqrt{s}\sqrt{1 + r}} \rg).	
		\end{equation}
		Now, using that $r = m/(2 \sqrt{s n})$, we have the bound
		\begin{equation}
		\label{E:rnn}
		\frac{r n^{1/6}}{16\sqrt{s(1 + r)}} \ge \frac{r n^{1/6}}{16\sqrt{s}(1 + r/2)} = \frac{m n^{1/6}}{32 s\sqrt{n} + 8 m\sqrt{s}}.
		\end{equation}
		Using that $s \le cn^{-1/3}$, there exists a constant $c'>0$ such that the right hand side of \eqref{E:rnn} is bounded below by
		$$
		\frac{c' m n^{1/6}}{n^{1/6} + m} \ge \frac{c'}2 (m \wedge n^{1/6}).
		$$
		Therefore \eqref{E:p-tobd} is bounded above by
		$$
		\p\lf(W_k^n(1) -2\sqrt{n} < - \frac{c'}2 mn^{-1/6} (m \wedge n^{1/6}) \rg).
		$$
		This is bounded above by $c_k e^{-d_km^{2}}$ for all $m$ by Theorem~\ref{T:top-bd}, giving that~\eqref{E:splat} is also bounded above by
		$
		c_k e^{-d_km^{2}}
		$
		for any $s \in [0, cn^{-1/3}]$. This in turn bounds~\eqref{E:P-cc}, completing the proof.
	\end{proof}

	\begin{proof}[Proof of Theorem~\ref{T:CH-i}]
		Tightness of the finite-dimensional distributions of $\scrC^n$ is well-known. It follows from tightness of the distribution of $\scrC^n(t)$ for every fixed $t$. This can be obtained, for example, from Theorem~\ref{T:top-bd} and Brownian scaling. Tightness of the whole ensemble $\scrC^n$ then follows from Proposition~\ref{P:dyson-tails} and the Kolmogorov-Centsov theorem (\citep{kallenberg2006foundations}, Corollary 16.9).
	\end{proof}
	
	We end this section by showing how Proposition~\ref{P:dyson-tails} can be used in conjunction with Theorem~\ref{T:top-bd} and Lemma~\ref{L:levy-est} to prove a type of `law of the iterated logarithm' for the top line in Dyson's Brownian motion. Beyond being an interesting consequence of Proposition~\ref{P:dyson-tails}, this result is necessary for the construction of the directed landscape in~\cite{DOV}.
	
	
	To prove this law, we first state a simple corollary of Proposition~\ref{P:dyson-tails} and Lemma~\ref{L:levy-est}.
	
	\begin{corollary}
		\label{C:mod-cont}
		For every $k$, there exist constants $c_k$ and $d_k$ such that for all $t > 0, m > 0$ and all $n \ge 1$ we have
		\begin{equation}
		\label{E:probprob}
		\prob\bigg(\max_{s \in [0, tn^{-1/3}]} \lf|W^n_k(t+ s) - W^n_k(t) - \frac{s \sqrt{n}}{\sqrt{t}} \rg| \ge m \sqrt{t} n^{-1/6} \bigg) \le c_k e^{-d_k m^{3/2}}.
		\end{equation}
		
	\end{corollary}
	
	\begin{proof}
		Proposition~\ref{P:dyson-tails} allows us to apply  Lemma~\ref{L:levy-est} with $d =1, \al_1 = 1/2, \beta_1 = 3/2, r_1 = b_1 = tn^{-1/3}$ to the process $W^n_k$ on the interval $[0, tn^{-1/3}]$. This gives that there exists a random constant $C > 0$ with $\p(C > m) \le c_k e^{-d_k m^{3/2}}$ for all $m$ such that
		$$
		\lf|W^n_k(t+ s) - W^n_k(t) - \frac{s \sqrt{n}}{\sqrt{t}} \rg| \le C\sqrt{s} \log^{2/3} \lf( \frac{2 tn^{-1/3}}{s}\rg)
		$$
		for all $s \in [0, tn^{-1/3}]$.
		Using that $\log^{2/3}(y) \le \sqrt{y/2}$ for all $y > 0$, we get that the right hand side above is bounded above by $C \sqrt{t}n^{-1/6}$ for all $s \in [0, tn^{-1/3}]$. The bound in \eqref{E:probprob} follows since $\p(C > m) \le c_k e^{-d_k m^{3/2}}$.
	\end{proof}
	
	\begin{prop}
		\label{P:cross-prob}
		There exist positive constants $b, c,$ and $d$ such that for all $m > 0$ and $n \ge 1$, the probability that
		\begin{align*}
		W^n_1(t) \le 2 \sqrt{nt} + \sqrt{t} n^{-1/6}[m + b\log^{2/3}(n^{1/3} \log(t \;\vee \;t^{-1}) + 1)] \;\; \forall t \in [0, \infty)
		\end{align*}
		is greater than or equal to $1- c e^{-d m^{3/2}}$.
		Here $x \vee y = \max(x,y)$.
	\end{prop}
	
	Note that in Proposition~\ref{P:cross-prob} we get a $2/3$-power in the outer logarithm rather than the power of $1/2$ seen in the usual law of the iterated logarithm. This is due to the fact that the top tail of the top line of Dyson's Brownian has a Tracy-Widom $3/2$-exponent. Throughout the proof, $b, c,$ and $d$ are constants that may change from line to line.
	\begin{proof}
		Set $a_i = (1 + n^{-1/3})^i$ for $i \in \Z$.
		By Theorem~\ref{T:top-bd} and Brownian scaling, for all $i \in \Z$ we have that
		\begin{align*}
		\p\Big(W^n_1(a_i) \ge 2 \sqrt{n}\sqrt{a_i} + [m + b\log^{2/3}(|i| + 1) &]\sqrt{a_i}n^{-1/6} \Big) \le c(|i|+1)^{-db}e^{-dm^{3/2}}.
		\end{align*}
		By Corollary~\ref{C:mod-cont}, we can extend this to
		\begin{equation}
		\label{E:Wn111}
		\begin{split}
		\p\bigg(\exists& t \in [a_i, a_{i+1}] : W^n_1(t) \ge 2 \sqrt{n}\sqrt{a_i} + \frac{(t- a_i)\sqrt{n}}{\sqrt{a_i}} \\
		&+ [m + b\log^{2/3}(|i| + 1) ]\sqrt{a_i}n^{-1/6} \bigg) \le c(|i|+1)^{-db}e^{-dm^{3/2}}
		\end{split}
		\end{equation}
		as well. Now, by Taylor expansion and the definition of the sequence $a_i$, for $t \in [a_i, a_{i+1}]$ we have that
		$$
		2 \sqrt{n}\sqrt{a_i} + \frac{(t- a_i)\sqrt{n}}{\sqrt{a_i}} \le 2 \sqrt{n}\sqrt{t} + \frac{n^{-1/6}}4, \;\; \sqrt{a_i} \le 2\sqrt{t} \; \mathand \;\; |i| \le 4 n^{1/3} \log(t \vee t^{-1}).
		$$
		Therefore
		\begin{align*}
		\p\bigg(\exists t \in [a_i, a_{i+1}] : W^n_1(t) &\ge 2 \sqrt{n}\sqrt{t} + [m + b\log^{2/3}(n^{1/3} \log(t \;\vee \;t^{-1}) + 1)]  \bigg)
		\end{align*}
		satisfies the same bound as~\eqref{E:Wn111}. A union bound then proves the proposition.
	\end{proof}

	\section{Properties of the Airy point process}
	\label{S:airy-point}
	
	In this section, we prove a few basic properties about the distribution of the points
	$$
	(\scrL_1(0), \scrL_2(0), \dots).
	$$
	This sequence of points is known as the \textbf{Airy point process}. It is determinantal with locally trace class kernel given by Equation~\eqref{E:airy-kernel} in the case $s=t$.  To simplify notation in the following lemmas, we write $A_i = \scrL_i(0)$, and for $a \in \R$, define
	$$
	N_a = \card {\{ i \in \N : A_i \in [-a, \infty) \}}.
	$$
	We first recall facts about the expected location of the $k$th point in the Airy point process, the expectation and variance of $N_a$, and tail bounds for the location of $A_1$. The facts about expectation can be easily derived from standard formulas for the Airy point process density (i.e. see formula 1.17 in~\citep{soshnikov2000gaussian} and discussion thereafter). The variance bound is more involved, and is proven as Theorem 1 in~\citep{soshnikov2000gaussian}. For this next lemma and throughout this section, we define $$\kappa  = (3\pi/2)^{2/3}.$$
	The tail bounds on $A_1$ are well-known. See, for example, Exercise 3.8.2 in \cite{anderson2010introduction} (alternately, they follow from the tail bounds on the prelimit in Theorem~\ref{T:top-bd}).
	
	\begin{lemma}
		\label{L:var-airy} \hfill
		\begin{enumerate}[label=(\roman*)]
			\item $\expt A_k = -\kappa k^{2/3} + O(1)$ as $k \to \infty$ and $\expt N_a = (a/ \kappa)^{3/2} + O(1)$ as $a \to \infty$.
			\item There exist constants $c_1$ and $c_2$ such that for all $a \ge 1$, we have $\Var (N_a) \le c_1 \log a + c_2$.
			\item There exists constants $c, d > 0$ such that for all $a > 0$, we have $\p(A_1 > a) \le c\exp(- da^{3/2})$ and $\p(A_1 < - a) \le c\exp(- da^{3})$.
		\end{enumerate}
	\end{lemma}
	
	As a quick consequence of the above lemma and the determinantal structure of the Airy point process, we can bound the fluctuations of the number of Airy points in an interval.
	
	\begin{lemma}
		\label{L:box-dev} There exists constants $\al, c > 0$ such that for every $x \in \R$ and $b > \al \log (x \vee 2)$, we have
		$$
		\p\lf(\lf|N_x- (x_+/\kappa)^{3/2} \rg| \ge b \rg) \le ce^{-b}.
		$$
		Here $x_+ =\max(x,0)$.
	\end{lemma}
	
	\begin{proof}
		By the monotonicity of $N_x$ in $x$, it is enough to prove the lemma when $x \ge 0$.
		The number of points in any interval in a determinantal point process with a locally trace class kernel is equal in distribution to a sum of independent Bernoulli random variables (see~\citep{hough2009zeros}, Theorem 4.5.3). Therefore Bernstein's inequality gives that
		$$
		\p(|N_x-\expt N_x| \ge b) \le 2 \exp\left(\frac{-\frac{1}2b^2}{\Var(N_x) + \frac{1}3 b} \right).
		$$
		Applying the bounds on $\expt N_x$ and $\Var(N_x)$ from Lemma~\ref{L:var-airy} completes the proof. 
	\end{proof}
	
	We also record a corollary which translates the above lemma into a bound on the Airy point locations.
	
	\begin{corollary}
		\label{C:point-locations}
		There exist $c, \beta > 0$ such that for all $i \in \N$ and $m > \beta \log (i + 1)$, we have that
		$$
		\p(|A_i + \kappa i^{2/3}| \ge mi^{-1/3}) \le ce^{-m/5}.
		$$
	\end{corollary}
	
	In the proof $c$ is a constant that may change from line to line.
	\begin{proof}
		Fix $m > 0$ and $i \in \N$ and let $x_m = \kappa i^{2/3} + mi^{-1/3}$. We have that
		\begin{equation}
		\label{E:A-bd}
		\p(A_i < -x_m) = \p(N_{x_m} < i) = \p\lf( (x_m/\kappa)^{3/2} - N_{x_m} >  (x_m/\kappa)^{3/2} - i\rg).
		\end{equation}
		We have the bound
		$$
		(x_m/\kappa)^{3/2} - i > \frac{3}{2\kappa } m > \frac{m}5.
		$$
		Moreover, letting $\al$ be as in Lemma~\ref{L:box-dev}, there exists a $\beta > 0$ such that
		$$
		\frac{m}5 > \al \log (x_m \vee 2)
		$$
		whenever $m > \beta \log (i + 1)$. Applying Lemma~\ref{L:box-dev} then shows that
		$$
		\p(A_i + \kappa i^{2/3} < - mi^{-1/3}) \le ce^{-m/5}.
		$$
		Now let $y_m = \kappa i^{2/3} - mi^{-1/3}$ and let $y_{m, +} = \max(y_m, 0)$. Observe that
		\begin{equation}
		\label{E:A-bd-2}
		\p(A_i > -\kappa i^{2/3} + m i^{-1/3}) = \p(N_{y_m} > i) = \p\lf(N_{y_m} - (y_{m, +}/\kappa)^{3/2} > i - (y_{m, +}/\kappa)^{3/2}\rg).
		\end{equation}
		We have that
		$$
		i - (y_{m, +}/\kappa)^{3/2} \ge \min (i, \kappa ^{-1} m) \ge \min(i, m/5).
		$$
		Again letting $\al$ be as in Lemma~\ref{L:box-dev}, there exists a $\beta > 0$ such that right hand above is bounded below by $\al \log (y_m \vee 2)$ whenever $5i \ge m > \beta \log i$. Applying Lemma~\ref{L:box-dev} then gives
		$$
		\p(A_i + \kappa i^{2/3} < - mi^{-1/3}) \le ce^{-m/5},
		$$
		for $5i \ge m > \beta \log i$.
		When $m \ge 5i$, by the standard Tracy-Widom tail bound (Lemma~\ref{L:var-airy} (iii)), we have that
		$$
		\p(A_i > -\kappa  i^{2/3} + mi^{-1/3}) \le \p(A_1 > -\kappa  i^{2/3} + mi^{-1/3}) \le ce^{-\frac{4}3 (2mi^{-1/3}/5)^{3/2}} \le c e^{-m/5}.
		$$
		Here we have used that $5 - \kappa  > 2$, so $-\kappa  i^{2/3} + mi^{-1/3} \ge 2mi^{-1/3}/5$.
	\end{proof}
	%
	%
	
	We now turn our attention to bounding the probability that Airy points are close together.
	For a point process $\Pi$ on $\R$, we say that a point $x \in \Pi$ is $\delta$-jammed if there is a point $y \in \Pi$ such that $|x -y| \le \delta$. We start with a proposition bounding the number of $\delta$-jammed points in general determinantal processes.
	
	\begin{prop}
		\label{P:de-jam-gen}
		Consider  a deteminantal point process $\Pi$ on an interval $[a, a + \ell]$ with a $C^2$ kernel $K$. Suppose that there exists a constant $b$ such that for all $(x, y) \in [a, a + \ell]^2$,
		$$
		|K(x, y)| \le b, \quad \lf|\del_x K(x, y)\rg| \le b^2, \quad \lf|\del_y K(x, y)\rg| \le b^2, \quad \mathand \quad \lf|\del_{x, y} K(x, y)\rg| \le b^3.
		$$
		Let $L$ be the number of $\delta$-jammed points in $\Pi$. Then for every $n \in \mathbb{N}$, we have that
		$$
		\mathbb{E} {\lfloor L/3 \rfloor \choose n} n! \le (2nb^{4}\delta^{3} \ell)^n.
		$$
		Here and throughout the remainder of the paper we use the convention that ${m \choose n} = 0$ whenever $n > m$.
	\end{prop}
	
	%
	
	To prove Proposition~\ref{P:de-jam-gen}, we start with a simple combinatorial lemma.
	
	\begin{lemma}
		\label{L:k-tuples}
		Let $\Pi$ be a point process on $[a, a + \ell]$ and let $L$ be the number of $\delta$-jammed points in $\Pi$. Let $R$ be the number of $2n$-tuples $(x_1, y_1, \dots, x_n, y_n)$ of distinct elements of $\Pi$ such that $|x_i-y_i|\le\delta$ for all $i \in \{1, \dots, n\}$. We have that
		$$R\ge {\lfloor L/3 \rfloor \choose n }n! 2^n.$$
		Here we use the convention that ${k \choose n} = 0$ if $k < n$.
	\end{lemma}
	
	\begin{proof}
		Let $Z = \{z_1 < z_2 < \dots < z_L\}$ be the set of $\delta$-jammed points. We can construct a partial matching on $Z$ via the following greedy algorithm. At each step $i$, we match $z_i$ to $z_{i+1}$ if $z_i$ was not matched to $z_{i-1}$ and if $|z_i - z_{i+1}| \le \delta$. This algorithm produces a matching with at least $\lfloor L/3 \rfloor$ pairs $(z_i, z_{i+1})$. By counting the assignments of $(x_i, y_i)$ among the pairs in this matching, we get the desired lower bound.
		%
		%
		%
	\end{proof}
	
	%
	%
	%

	\begin{proof}[Proof of Proposition~\ref{P:de-jam-gen}]
		Let $R$ be as in Lemma~\ref{L:k-tuples}. We have that
		$$
		\expt R = \int_S \det M d \mu, \quad \text{where} \quad S = \lf\{(x_1, y_1, \dots, x_n, y_n)\in [a,a+\ell]^{2n} : |x_i-y_i|\le\delta \rg\}
		$$
		and $\mu$ is Lebesgue measure (see Remark 1.2.3, Definition 4.2.1 in~\cite{hough2009zeros}).
		In the above formula, $M$ is a $2n \X 2n$ matrix consisting of $n^2$ $2\times 2$ blocks, where each block is of the form
		\[
		M_{i, j} =
		\begin{bmatrix}
		K(x_i,x_j) & K(x_i,y_j) \\
		K(y_i,x_j) &K(y_i,y_j) \\
		\end{bmatrix}.
		\]
		We first bound $\det M$ on the set $S$. To do this, we will compute $\det B M B^t$, where $B$ is a block diagonal matrix consisting of $n$ $2 \times 2$ blocks of the form
		\[
		A=
		\begin{bmatrix}
		1/2 & 1/2 \\
		(b\delta)^{-1} & -(b\delta)^{-1} \\
		\end{bmatrix}.
		\]
		We can calculate the blocks of $BMB^t$ by computing that
		\[
		A
		\begin{bmatrix}
		e & f \\
		g & h \\
		\end{bmatrix}
		A^t =
		\begin{bmatrix}
		\dfrac{e+f+g+h}4 \quad & \quad \lf(\dfrac{e-f}{\de}+\dfrac{g-h}{\de}\rg)/2b \\[2ex]
		\left(\dfrac{e-g}{\delta}+\dfrac{f-h}{\delta}\right)/2b \quad &  \quad \lf(\dfrac{(e -g)}{\delta} - \dfrac{(f -h)}{\delta}\rg)/(\de b^2) \\[2ex]
		\end{bmatrix}.
		\]
		%
		Substituting in the entries of $M_{i, j}$ for $e, f, g,$ and $h$ above, the $(1, 1)$-entry of the resulting matrix $AM_{i, j} A^t$ is an average of values of $K$. The $(1, 2)$-entry is an average of difference quotients of $K$, multiplied by $(\delta b)^{-1}(x_j - y_j)$, and similarly the $(2, 1)$-entry  is an average of difference quotients of $K$, multiplied by $(\delta b)^{-1}(x_i - y_i)$. Finally, the $(2, 2)$-entry is a
		second difference quotient, multiplied by $(\delta b)^{-2}(x_i - y_i)(x_j - y_j)$.
		
		
		The mean value theorem then implies that all entries of $BMB^t$ are bounded by $b$ on $S$ since $|x_j - y_j| \le \delta$ for all $i \in \{1, \dots, n\}$ on this set. Therefore
		$$
		\det(M) = \frac{\det(BMB^t)}{(\det B)^2} = (b\delta)^{2n}\det(BMB^t) \le
		(b\delta)^{2n}\sqrt{2nb^2}^{\,2n}.
		$$
		The last inequality follows from Hadamard's inequality and the fact that the length of each of the columns of $BMB^t$ is bounded by $\sqrt{2 n b^2}$.
		Combining this inequality with the fact that the volume of $S$ is less than $(2\delta\ell)^{n}$ shows that $\expt R \le (4nb^4 \de^3\ell)^n$. Lemma~\ref{L:k-tuples} then completes the proof of Proposition~\ref{P:de-jam-gen}.
	\end{proof}
	%
	%
	%

	\begin{corollary}
		\label{C:de-jam-airy}
		There exists a constant $c$ such that for any $a \ge 1, \; \ell > 0$, and $n \in \N$, the number $L$ of $\delta$-jammed points in the Airy point process in the interval $[-a, -a + \ell]$ satisfies
		$$
		\mathbb{E} {\lfloor L/3 \rfloor \choose n} n! \le (cna^{2}\delta^{3} \ell)^n.
		$$
	\end{corollary}
	%
	%
	\begin{proof}
		Recall from~\eqref{E:airy-kernel} that we have the following formula for the kernel of the Airy point process:
		\begin{equation}
		\label{E:Kxy}
		K(x,y)=\int_0^\infty{\Ai(x+\la)\Ai(y+\la)d\la}.
		\end{equation}
		
		To bound $K$ and its derivatives, we use the following bounds that hold for all $x \ge 1$ (see~\citep{abramowitz1972handbook}, formulas 10.4.59-10.4.62).
		\begin{align}
		\label{E:airy-bds}
		|\sqrt{x}\Ai(-x)|, \; |\Ai'(-x)|\le& \;cx^{1/4} \quad \mathand \quad |\sqrt{x}\Ai(x)|, \; |\Ai'(x)|\le cx^{1/4}e^{-(2/3)x^{3/2}}.
		\end{align}
		Here $c$ is a positive constant. Also note that $\Ai(x)$ and $\Ai'(x)$ are bounded on $[-1,1]$ by continuity.
		Since the Airy function is analytic and decays exponentially fast as $x \to \infty$, we can differentiate under the integral sign in~\eqref{E:Kxy} to get formulas for $\partial_x K, \partial_y K$ and $\partial_{x, y} K$. By the bounds in~\eqref{E:airy-bds}, we get that
		$$
		|K(x,y)|\le m^{1/2},\quad |\partial_x K(x,y)|\le m, \quad |\partial_{x,y}K(x,y)|\le m^{3/2},
		$$
		where $m=c'\max(-x,-y,1)$ for a constant $c'$. Applying Proposition~\ref{P:de-jam-gen} with $b = c'a^{1/2}$ finishes the proof.
	\end{proof}

	As a consequence of Corollary~\ref{C:de-jam-airy}, we can get the following tail bound on the number of $\de$-jammed points in the Airy point process in a given interval. This next proposition is a restatement of Proposition~\ref{P:mean-conc-i}.
	\begin{prop}
		\label{P:mean-conc}
		There exists a constant $d > 0$ such that for all $a \ge 1, \; \ell, m > 0$, and $\de > 0$,  the number of $\de$-jammed points $L$ in the Airy point process in the interval $[-a, -a + \ell]$ satisfies
		$$
		\prob (L > m a^2 \de^3 \ell) \le 2\exp \lf(-d m\lf(\frac{a^2\de^3\ell}{a^2 \de^3 \ell + 1}\rg) \rg).
		$$
	\end{prop}
	
	\begin{proof}
		Observe that for every $n, m \in \N$ we have
		\begin{equation}
		\label{E:mfact}
		m^n \le 4^n {\floor{m/3} \choose n} n! + (12n)^n.
		\end{equation}
		To see this, first note that $m^n \le (12n)^n$ if $m \le 12n$. For $m \ge 12n$, we have that
		$$
		4^n {\floor{m/3} \choose n} n! = 4^n \floor{m/3}(\floor{m/3} - 1) \cdots (\floor{m/3} - n + 1).
		$$
		Since $m \ge 12n,$ we have that $\floor{m/3} - n + 1 \ge m/4$, and so the right hand side above is greater that $4^n(m/4)^n = m^n$, yielding \eqref{E:mfact} in this case.
		Therefore by Corollary~\ref{C:de-jam-airy} and Fubini's Theorem, for any $b > 0$ we have that
		$$
		\expt e^{bL} = \sum_{n=0}^\infty \frac{b^n \expt L^n}{n!} \le \sum_{n=0}^\infty \lf(\frac{(12nb)^n}{n!} + \frac{(4cna^2 \de^3 \ell b)^n}{n!} \rg).
		$$
		Here the constant $c$ is as in Corollary~\ref{C:de-jam-airy}. Therefore there is a universal constant $d > 0$ such that with $b=d/(a^2 \de^3 \ell + 1)$ we have $\expt e^{bL}\le 2$.
		Applying Markov's inequality to the event $e^{bL}>e^{mba^2\delta^3\ell}$ completes the proof.
	\end{proof}
	
	In order to prove the bridge representation of the parabolic Airy line ensemble $\scrL$, we first need to define associated graphs that record which points in $\scrL$ are $\de$-jammed.
	
	\begin{definition}
		\label{D:bridge-graph}
		Fix $t > 0$ and $\ell \in \N$. Define $s_i = it/\ell$ for all $i \in \{0, 1, \dots, \ell\}$. For $k, \de > 0$, we define a random graph $G_k(t, \ell, \de)$ on the set
		$$
		S_k(\ell) = \{1, \dots, k\} \X \{1, \dots, \ell \},
		$$
		where the points $(i, j)$ and $(i + 1, j)$ are connected if either
		$$
		|\scrL_i(s_{j-1}) - \scrL_{i+1}(s_{j-1})| \le \de \quad \mathor \quad |\scrL_i(s_j) - \scrL_{i+1}(s_j)| \le \de.
		$$
		Let $M_k(t, \ell, \de)$ be the size of the largest component of $G_k(t, \ell, \de)$.
	\end{definition}
	
	The second important consequence of Corollary~\ref{C:de-jam-airy} gives a bound on the size of components in $G_k(t, \ell, \de)$. In other words, it allows us to bound the size of long chains of $\de$-jammed points in the parabolic Airy line ensemble.
	
	\begin{prop}
		\label{P:delta-chains} There exists $c>0$ so that for every $\de \in (0, 1]$, $\ell, m, k \in \N$ with $k \ge 2$ we have that
		\begin{equation}
		\label{E:m-want}
		\prob (M_k(t, \ell, \de k^{-1/3}) \ge m) \le \ell k [c m \log (k/\de)]^m \de^{3\floor{m/6}}.
		\end{equation}
	\end{prop}
	
	Proposition~\ref{P:delta-chains-i} follows as a special case of Proposition~\ref{P:delta-chains} in the case when $\de = k^{-\ga/4}, \ell = \ceil{t k^{2/3 + \ga}}$ for fixed $t, \ga$. With these parameter choices, for fixed $m$, as $k \to \infty$ the right hand side of \eqref{E:m-want} is equal to
	\begin{equation}
	\label{E:mkkk}
	O(\log^m(k)) k^{5/3 + \ga(1 - 3 \floor{m/6}/4)}.
	\end{equation}
	The exponent in $k$ is always negative if we take $m = \ceil{14(1 + 1/\ga)}$, and hence \eqref{E:mkkk} tends to $0$ with $k$ for this choice of $m$.
	
	
	In the proof, $c$ is a constant that may change from line to line.
	
	\begin{proof}
		Fix $m, k, \ell \in \N$ and $\de \in (0, 1]$ with $k \ge 2$. We may assume that $6 \le m \le k$ since by increasing $c$, the statement is trivial for $m < 6$, and the component sizes in $G_k(t, \ell, \de)$ are deterministically bounded above by $k$. If $M_k(t, \ell, \de k^{-1/3}) \ge m$, then there must be some $(i, j)$ with $i \ge m$ in the set
		$$
		\tilde S_k(\ell) = \{1, \dots, k\} \X \{0, \dots, \ell\}
		$$
		such that there are at least $m/2$ points in the set
		$$
		\{\scrL_{i - m + 1}(s_j), \dots, \scrL_{i-1}(s_j),\scrL_{i}(s_j) \}
		$$
		that are $\de k^{-1/3}$-jammed. Let $E_{k, b}$ be the event where
		$$
		|\scrL_i(s_j) + s_j^2 + \kappa i^{2/3}|<b i^{-1/3}\log (k + 1) \quad \text{for all} \quad (i, j) \in \tilde S_k(\ell).
		$$
		On the event $E_{k, b}$, we have that
		$$
		\{\scrL_{i - m + 1}(s_j) + s_j^2, \dots, \scrL_{i-1}(s_j) + s_j^2, \scrL_{i}(s_j) + s_j^2 \} \sset [-a_i, -a_i + \ell_i],
		$$
		for all $(i, j) \in \tilde S_k(\ell)$,	where
		$$
		a_i = \kappa i^{2/3} + b i^{-1/3}\log (k + 1) \quad\mathand \quad \ell_i = \kappa mi^{-1/3} + 2b (i-m + 1)^{-1/3}\log(k + 1).
		$$
		Therefore letting $J_{i, j}$ be the number of $\delta k^{-1/3}$-jammed points in the point process $\{\scrL_r(s_j) : r \in \N\}$
		in the set $[-a_i, -a_i + \ell_i]$, a union bound implies that
		\begin{equation}
		\label{E:M-L}
		\prob(M_k(t, \ell, \de) \ge m) \le \prob(E_{k, b}^c) + \sum_{(i, j) \in \tilde S_k(\ell)} \prob(J_{i, j} \ge m/2).
		\end{equation}
		We can bound $\p(E^c_{k, b})$ and $\p(J_{i, j} \ge m/2)$ by using that each of the processes
		$$
		\{\scrL_i(s_j) + s_j^2 : i \in \N\}
		$$
		is an Airy point process.
		Set $b = \max(\beta, 15 \gamma \floor{m/6})$, where $\beta$ is the constant from Corollary~\ref{C:point-locations} and $\ga = - \log_{k} \de$. A union bound implies that
		\begin{equation}
		\label{E:E-prob}
		\prob (E_{k, b}^c) \le c (\ell + 1)k^{1 - b/5} \le c \ell k^{1 - 3\gamma \floor{m/6}}.
		\end{equation}
		It remains to bound $\prob(J_{i, j} \ge m/2)$. By Markov's inequality,
		\begin{equation}
		\label{E:Lij1}
		\prob(J_{i, j} \ge m/2) \le \prob\lf( {\floor{J_{i, j}/3} \choose \floor{m/6}} > 0\rg) \le \frac{1}{\floor{m/6} !} \expt {\floor{J_{i, j}/3} \choose \floor{m/6}} \floor{m/6}!
		\end{equation}
		In the first inequality in \eqref{E:Lij1}, we used that $m \ge 6$. 
		By applying Corollary~\ref{C:de-jam-airy} and using the bound $n! \ge (n/e)^n$ for all $n$, the right hand side of \eqref{E:Lij1} is bounded above by
		\begin{equation}
		\label{E:Lij}
		(c a_i^2 \de^3 k^{-1} \ell_i)^{\floor{m/6}},
		\end{equation}
		where $c$ is a universal constant. Now, for a universal constant $c$, we have the bounds
		$$
		b \le c(1 + m\ga) \quad \mathand \quad (i - m + 1)^{-1/3} \le i^{-1/3} + \frac{(m-1)}{i-1},
		$$
		which gives that
		\begin{align*}
		&\ell_i \le cm(1 + \ga)\lf(i^{-1/3} + \frac{(m-1)}{i-1}\rg) \log(k + 1) \quad \mathand \\
		&a_i \le \kappa i^{2/3} + c(1 + m\ga) i^{-1/3}\log(k + 1).
		\end{align*}
		Using this we get that for every $i \in [m, k] \cap \N$, that
		$$
		a_i^2\ell_i \le  c m^{4/3} (1 + \ga^3) k \log^3 (k + 1)
		$$
		Using that $\ga = - \log_k \de$ and that $k \ge 2$ so $\log (k+1) \le c \log k$, we can bound the right hand side of~\eqref{E:Lij} above by
		$$
		\lf[c m^{4/3} \log^3 (k/\de) \de^3\rg]^{\floor{m/6}}.
		$$
		Combining this with the bound in~\eqref{E:E-prob} bounds the right hand side of~\eqref{E:M-L}. This completes the proof.
	\end{proof}
	
	\section{Preliminary modulus of continuity for parabolic Airy lines}\label{S:moduli}
	
	In this section, we will obtain a modulus of continuity estimate for the parabolic Airy line ensemble. This will be improved later as a consequence of the bridge representation. However, we need to prove a preliminary estimate in order to prove that theorem. When combined with the pointwise bounds in Corollary~\ref{C:point-locations}, this modulus of continuity estimate gives a bound on the maximum of the $k$th Airy line over an interval. This bound will be necessary for showing that boundary conditions don't propagate upwards when applying the Brownian Gibbs property in a small region.
	
	
	For this modulus of continuity, we will need a tail bound on the difference between values in the parabolic Airy line ensemble at two distinct times. The method of proof is a limiting version of the method used in Proposition~\ref{P:dyson-tails}. It is cleaner than that proof because we can appeal to the stationarity of the Airy line ensemble.
	
	\begin{lemma}
		\label{L:airy-tails}
		There are constants $c, d > 0$ such that for every $t \in \R, s \in (0, 1), a > 0,$ and $k \in \N$, the $k$th line $\scrL_k$ in the parabolic Airy line ensemble satisfies
		\begin{equation}
		\label{E:lemma61}
		\p(|\scrL_k(t) + t^2 - \scrL_k(t + s) - (t+s)^2| > a\sqrt{s}) \le e^{ck - da^2}.
		\end{equation}
		Moreover, when $k = 1$ the above bound holds for all $t \in \R, s \in (0, \infty)$ and $a > 0$.
	\end{lemma}
	
	\begin{proof} First, suppose that $a > 7 s^{3/2}$. Set $b = a\sqrt{s} + s^2$. By stationarity of the Airy line ensemble $\scrL(t) + t^2$ and the equality in distribution $\scrL(\cdot) \eqd \scrL(- \; \cdot)$, we have that
		\begin{align*}
		&\p(|\scrL_k(t) + t^2 - \scrL_k(t + s) - (t+s)^2| > a\sqrt{s}) \\
		= 2&\p(\scrL_k(0) - \scrL_k(s) - s^2 > a\sqrt{s}) \\
		= 2&\p(\scrL_k(0) - \scrL_k(s) > b).
		\end{align*}
		Now set $r = b/(8s)$.
		Since $a > 7s^{3/2}$, we have that $r > s$. By the Brownian Gibbs property for $\scrL$, conditionally on the set
		$$
		\{\scrL_i(0), \scrL_i(r) : i \in \{1, \dots, k\} \} \cup \{\scrL_{k+1}(q) : q \in [0, r]\},
		$$
		the restriction $\scrL|_{\{1, \dots, k\} \X[0, r]}$ is equal in distribution to $k$ Brownian bridges $(\tilde{B}_1, \dots, \tilde{B}_k)$ on $[0, r]$ of variance $2$ with endpoints given by $\scrL_i(0)$ and $\scrL_i(r)$ for all $i \in [1, k]$, conditioned to avoid  $\scrL_{k+1}$. By Lemma~\ref{L:monotone-gibbs}, $(\tilde{B}_1, \dots, \tilde{B}_k)$ stochastically dominates the sequence
		$$
		(B_1 + L, \dots, B_k + L),
		$$
		where $(B_1, \dots, B_k)$ is a Brownian $k$-melon with variance $2$ on $[0, r]$, and $L$ is the affine function with $L(0) = \scrL_k(0)$ and $L(r) = \scrL_k(r)$.
		We have that
		\begin{align*}
		\p(\scrL_k(0) - \scrL_k(s)  > b) &\le \p([B_k + L](0) - [B_k + L](s)  > b) \\
		&\le \p\lf(L(0) - L(s) > \frac{b}2 \rg)  + \p\lf(B_k(s) - B_k(0)< -\frac{b}2 \rg).
		\end{align*}
		We can then relate the first term above to the values of $\scrL_k$ at the times $0$ and $r = b/(8s)$. This gives
		\begin{align}
		\nonumber
		\p\lf(L(0) - L(s) \ge \frac{b}2 \rg) &= \p\lf(\scrL_k(0) - \scrL_k(r) \ge \frac{br}{2s} \rg) \\
		\nonumber
		&\le \p\lf(\scrL_k(r) \le -\kappa k^{2/3} - \frac{br}{4s}\rg) + \p\lf(\scrL_k(0) \ge -\kappa k^{2/3} + \frac{br}{4s}\rg) \\
		\label{E:Rkk}
		&=\p\lf(\scrL_k(r) + r^2 \le -\kappa k^{2/3} - \frac{b^2}{64s^2}\rg) + \p\lf(\scrL_k(0) \ge -\kappa k^{2/3} + \frac{b^2}{32s^2}\rg).
		\end{align}
		Here $\ka =  (3\pi/2)^{2/3}$, and for the final equality we have used that $r=b/(8s)$.
		We can bound the two probabilities in \eqref{E:Rkk} using Corollary~\ref{C:point-locations}, since $\scrL_k(r) + r^2$ is a $k$th Airy point for any $r$. Also, $\p\lf(B_k(s) - B_k(0) \le - b/2 \rg)$ can be bounded by Lemma~\ref{L:bridge-disc}. Combining these bounds gives that
		$$
		\p(\scrL_k(0) - \scrL_k(s)  > b) \le ce^{-db^2/s^2} + e^{ck -db^2/s}
		$$
		for constants $c$ and $d$. Using that $b \ge a\sqrt{s}$ proves the desired bound when $s \in (0, 1)$. For the case when $k = 1, s \ge 1,$ and $a > 7 s^{3/2}$ we can use the tail bounds on the Tracy-Widom random variable $\scrL_1(0)$ from Lemma~\ref{L:var-airy}(iii) instead of Corollary~\ref{C:point-locations} to bound~\eqref{E:Rkk}. This gives that
		$$
		\p(\scrL_1(0) - \scrL_1(s)  > b) \le ce^{-db^3/s^3} + e^{c -db^2/s}.
		$$
		Using the fact that $b \ge a\sqrt{s}$ and that  $a> 7s^{3/2}$ then gives the desired bound.
		
		Now suppose that $a < 7 s^{3/2}$. When $s < 1$, then $a < 7$, so by increasing $c$ we can guarantee that $e^{ck - da^2} > e^{c - 7d} \ge 1$, and the bound follows trivially. For the $k = 1$ and $s \ge 1$ case, Tracy-Widom tail bounds on the random variables $\scrL_1(t) + t^2$ and $\scrL_1(t+s) + (t+s)^2$ (Lemma~\ref{L:var-airy}(iii)) and a union bound imply that
		\begin{equation*}
		\p(|\scrL_1(t) + t^2 - \scrL_1(t + s) - (t+s)^2| > a\sqrt{s}) \le c e^{-d a^{3/2} s^{3/4}}
		\end{equation*}
		for some $c, d > 0$.
		Since $a < 7 s^{3/2}$, this is bounded above by $c e^{-d 7^{-1/2} a^{2}}$, completing the proof.
	\end{proof}
	
	As a consequence of Lemma~\ref{L:airy-tails}, we get the following modulus of continuity for the parabolic Airy line ensemble.
	
	\begin{theorem}
		\label{T:airy-mod}
		There are $c,d>0$ so that for any $k \in \N$, $a\in \R$, there exists a random constant $C_k$ so that
		$$
		|\scrL_k(t) + t^2 - \scrL_k(t + s) - (t+s)^2| \le C_k \sqrt{s} \log^{1/2}(2/s)
		$$
		for all $t, t + s \in [a, a + 1]$. Moreover
		$$
		\p(C_k > m) \le e^{ck - d m^2}  \qquad \text{for all }
		m > 0.$$
	\end{theorem}
	
	\begin{proof}
		By Lemma~\ref{L:airy-tails}, we can apply Lemma~\ref{L:levy-est} to the $k$th Airy line $\scrL_k(t) + t^2$ on the interval $[a, a+1]$ with $d=1, b_1 = r_1 = 1, \al_1=1/2,$ and $\beta_1 = 2$.
	\end{proof}
	
	Theorem~\ref{T:airy-mod-i} follows from Theorem~\ref{T:airy-mod} in the $a=0$ case.
	
	\begin{proof}[Proof of Theorem~\ref{T:airy-mod-i}] For $t, t + s \in [0, 1]$, we have the deterministic bound
		$$
		|\scrL_k(t) - \scrL_k(t+s)| \le |\scrL_k(t) + t^2 - \scrL_k(t+s) - (t + s)^2| + 4\sqrt{s}\log^{1/2}(2/s), 
		$$
		so
		$$
		|\scrL_k(t) - \scrL_k(t+s)| \le (C_k + 4) \sqrt{s} \log^{1/2}(2/s),
		$$
		with $C_k$ as in Theorem~\ref{T:airy-mod}. The constant $C_k + 4$ satisfies the same tail bounds as $C_k$ with a larger value of $c$ and a smaller value of $d$. This yields Theorem~\ref{T:airy-mod-i}.
	\end{proof}

	We record the the following corollary of Theorem~\ref{T:airy-mod}, which confines lines in the parabolic Airy line ensemble.
	
	\begin{corollary}
		\label{C:k-line-bd}
		There exist $\beta ,c> 0$ such that for any $k \in \N$, any $t > 0$, and any $m > \beta \log (k + 1)$, we have that
		$$
		\p\lf(\sup_{s \in [0, t]} |\scrL_k(s) + s^2 + \kappa k^{2/3}| > m k^{-1/3} \rg) \le (t + 1)ce^{-m/10}.
		$$
	\end{corollary}
	
	In the proof, $c, d,$ and $\beta$ are positive constants that may change from line to line.
	\begin{proof}
		It suffices to proves the statement for $t = 1$, as for larger $t$, we can simply use a union bound and the stationarity of $\scrL_k(s) + s^2$, and for smaller $t$, the statement follows by modifying the constant $c$.
		Let $\Pi = \{0, 1/k^{2}, 2/k^{2}, \dots, 1\}$.
		Letting $C_k$ be as in Theorem~\ref{T:airy-mod} for the interval $[0, 1]$, we have that
		\begin{equation*}
		\sup_{s \in [0, t]} |\scrL_k(s) + s^2 + \kappa k^{2/3}| \le \sup_{s \in \Pi} |\scrL_k(s) + s^2 + \kappa k^{2/3}| + C_k \frac{\log^{1/2}(2k^2)}k,
		\end{equation*}
		and so the probability in the statement of the corollary is bounded above by
		\begin{equation}
		\label{E:s-0t}
		\p\lf(\sup_{s \in \Pi} |\scrL_k(s) + s^2 + \kappa k^{2/3}| > \frac{2m k^{-1/3}}3\rg) + \p\lf(C_k > \frac{mk^{2/3}}{3\log^{1/2}(2k^2)}\rg).
		\end{equation}
		When $m > \beta \log (k + 1)$, the first term on the right hand side of~\eqref{E:s-0t} is bounded above by $c k^2 e^{-2m/15}$ by Corollary~\ref{C:point-locations} and a union bound. The second term on the right hand side of~\eqref{E:s-0t} is bounded above by $e^{ck - dkm^2}$ by Theorem~\ref{T:airy-mod}. This uses that for all $k \in \N$, $k^{2/3}/(3\log^{1/2}(2k^2)) \ge d k^{1/2}$ for some $d > 0$. Now, as long as $\beta$ is large enough, we have that
		$$
		ck^2 e^{-2m/15} + e^{ck - dkm^2}
		$$
		is bounded above above by $ce^{-m/10}$ for all $k \in \N, m > \beta \log (k+1)$. This gives the desired bound.
	\end{proof}
	
	Corollary~\ref{C:k-line-bd} immediately implies Theorem~\ref{T:lln-i}.
	\begin{proof}[Proof of Theorem~\ref{T:lln-i}]
		By stationarity of the Airy line ensemble $\scrL(t) + t^2$ and Corollary~\ref{C:k-line-bd}, we have
		$$
		\p\lf( \sup_{s \in [-t_k, t_k]} |\scrL_k(t) + t^2 + \ka k^{2/3} | > \ep \ell_k k^{-1/3} \rg) \le (2t_k + 1)ce^{-\ep \ell_k/10}
		$$
		for a constant $c>0$.
		Since $t_k = e^{o(\ell_k)}$ and $\log k = o(\ell_k)$, the right hand side above is summable in $k$ for every $\ep > 0$. Hence Theorem~\ref{T:lln-i} follows by the Borel-Cantelli lemma.
	\end{proof}
	
	\section{The bridge representation}\label{S:bridge}
	
	In this section we prove the bridge representation, Theorem~\ref{T:bridge-rep-i}. After proving the bridge representation, we prove a small proposition that shows that at most locations, the bridge representation samples bridges without any nonintersection condition.
	
	
	We begin with a definition. Let $t>0,\ell\in \N$ and let $s_j=jt/\ell$ as in Section~\ref{S:airy-point}. Recall that $G_k(t, \ell, \delta)$ is a random graph  on the set $\{1, \dots, k\} \X\{1, \dots, \ell\}$ where vertices $(i, j)$ and $(i + 1, j)$ are connected if either
	$$
	|\scrL_i(s_{j-1}) - \scrL_{i+1}(s_{j-1})| \le \de \qquad \mathor \qquad |\scrL_i(s_j) - \scrL_{i+1}(s_{j})| \le \de.
	$$
	\begin{definition}
		\label{D:bridge-rep} Let $k \in \N$. The \textbf{bridge representation} $\mathcal{B}^k(t, \ell, \de)$ of the parabolic Airy line ensemble is a line ensemble $(\scrB^k_1, \dots, \scrB^k_{2k})$ with domain $[0, t]$ constructed as follows. For every $(i, j) \in \{1, \dots, 2k\} \X \{1, \dots, \ell\}$, sample an independent Brownian bridge $B_{i, j}:[s_{j-1}, s_{j}] \to \R$ with
		$$
		B_{i, j} (s_{j-1}) = \scrL_{i}(s_{j-1}) \quad \mathand \quad B_{i, j} (s_j) = \scrL_{i}(s_j),
		$$
		where the bridges $B_{i, j}$ and $B_{i', j}$ are conditioned not to intersect if $(i, j)$ and $(i',j)$ are in the same component of $G_{2k}(t, \ell, \delta)$. We then define the $i$th line $\scrB^k_i:[0, t]
		\to \R$ of the line ensemble $\mathcal{B}^k(t, \ell, \de)$ by concatenating the bridges $B_{i, j}$. That is, $\scrB^k_i|_{[s_{j-1}, s_j]} = B_{i, j}$ for all $j \in \{1, \dots, \ell\}$.
	\end{definition}
	
	The next theorem is a restatement of our main theorem, Theorem~\ref{T:bridge-rep-i}. It shows that $\mathcal B$ and $\mathcal L$ are close in law when restricted to the set $\{1, \dots, k\} \X [0, t]$ in the appropriate parameter range.

	\begin{theorem}
		\label{T:bridge-rep} There exist constants $c, d > 0$ such that the following holds for all $k \ge 3$, $t > 0$, $\ga \in (c \log(\log k)/\log k, 2]$ and $\ell \ge tk^{2/3 + \ga}$. The total variation distance between the laws of $\scrB^k(t, \ell, k^{-1/3 - \ga/4})|_{\{1, \dots, k\} \X [0, t]}$ and $\scrL|_{\{1, \dots, k\} \X [0, t]}$ is bounded above by
		$$
		\ell e^{-d \ga k^{\ga/12}}.
		$$
	\end{theorem}
	
	Theorem~\ref{T:bridge-rep} can be thought of as giving a quantitative version of the Brownian Gibbs property. In particular, it allows us to apply the Brownian Gibbs property to large patches in $\scrL$ without having to worry about boundary conditions or long-range interactions between lines (e.g. for a fixed value of $\ga$, the components in $G_k(t_k, t k^{2/3 + \ga}, k^{-1/3 - \ga/4})$ are of bounded size as we take $k \to \infty$ by Proposition~\ref{P:delta-chains-i}).
	
	
	The lower bound on $\ga$ in Theorem~\ref{T:bridge-rep} is optimal up to lowering the constant $c$; having a nearly optimal parameter range is essential in applications (e.g. Theorem~\ref{T:mod-cont-i} and the application of Theorem~\ref{T:bridge-rep} in~\cite{DOV}). The upper bound on $\ga$ in Theorem~\ref{T:bridge-rep} is quite artificial and is chosen purely for the sake of proof. For any $\ga > c \log \log (k)/ \log k$, we can get a similar total variation bound, though the exponent $\ga/12$ will be replaced by $b \ga$ for a constant $b$ that decays as $\ga \to \infty$. The restriction that $k \ge 3$ is imposed purely to ensure that $\log \log k/ \log k > 0$. See the introduction for discussion about the heuristics behind the parameter ranges.
	
	
	
	The first step in the proof of Theorem~\ref{T:bridge-rep} is to show that when we apply the Brownian Gibbs property to $\scrL$ on a region of the form $\{1, \dots, 2k\} \X[0, k^{-2/3 - \ga}]$, that the lower boundary condition coming from the line $\scrL_{2k + 1}$ only affects nearby lines.
	
	\begin{lemma}
		\label{L:key-bridge}
		Define the $\sigma$-fields
		\begin{align*}
		\scrF(k, t)&= \sig\{\scrL_i(s) : (i, s) \in \{1, \dots, 2k\} \X\{0,t\} \;\text{ or }\; i = 2k+1, s \in [0, t]\} \quad \mathand \\
		\scrG&= \sig\{\scrL_i(s) : (i, s) \in \{1, \dots, 2k\} \X\{0,t\}\}.
		\end{align*}
		Then for every $k \ge 3$, $\ga \in [c \log \log k/ \log k, 2]$ and $t \in (0, k^{-2/3 - \ga}]$, there exists a $\scrG$-measurable random variable $J\in (k,2k]$ such that the following holds. Here $c > 0$ is a universal constant.
		
		
		Let $A = (A_J, \dots, A_{2k})$ be a sequence of $2k +1 - J$ random functions, such that conditionally on $\scrF(k, t)$, the sequence $A$ consists of $2k +1 - J$ independent Brownian bridges of variance $2$ with endpoints $A_i(0) = \scrL_i(0)$ and $A_i(t) = \scrL_i(t)$ conditioned not to intersect each other or the lower boundary $\scrL_{2k + 1}$. Let $L_{J-1}$ be the line with $L_{J-1}(s) = \scrL_{J-1}(s)$ for $s = 0, t$. Then with (unconditional) probability at least $1 - e^{- d \ga k^{\ga/12}}$, we have the following conditional probability bound:
		\begin{equation}
		\label{E:cond-J}
		\prob \Big(\min_{s\in[0, t]}\lf( L_{J-1}(s) - A_{J}(s) \rg) \ge k^{-1/3 - \gamma/4}\;\; \Big| \; \scrF(k, t)\Big) \ge 1 - e^{- d \ga k^{\ga/4}}
		\end{equation}
		Here $d$ is a universal constant.
	\end{lemma}
	
	See Figure~\ref{fig:LemmaPropogation} for an illustration of the lemma and the main idea of the proof.

	\begin{figure}%
		\centering
		\begin{subfigure}[t]{2.5cm}
			\centering
			\includegraphics[width=2.5cm]{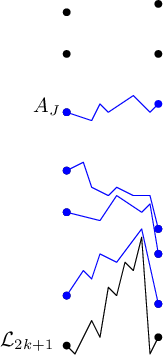}
			\caption{}
		\end{subfigure}
		\qquad \qquad \qquad \qquad
		\begin{subfigure}[t]{2.5cm}
			\centering
			\includegraphics[width=2.5cm]{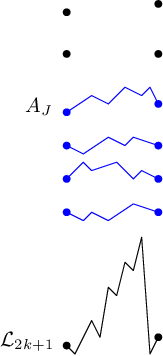}
			\caption{}
		\end{subfigure}
		\caption{Lemma~\ref{L:key-bridge} essentially shows that on the right scale, it is highly unlikely that the effect of the lower boundary condition is propagated to the top $k$ lines when applying the Brownian Gibbs property on the top $2k$ lines. More precisely, there exists a $J > k$ (the level of the top blue points) for which the bridge connecting the points at level $J$ stays well below the two black boundary points at level $J - 1$. This is illustrated in Figure~\ref{fig:LemmaPropogation}(a). Figure~\ref{fig:LemmaPropogation}(b) illustrates how this is proven. We shift some blue points up so that with high probability, independent Brownian bridges between all the blue points will not intersect each other or the lower boundary. This shifting process stabilizes by the point $J$. Therefore we can conclude that the bridge at level $J$ is stochastically dominated by a Brownian bridge off of a low probability event.}%
		\label{fig:LemmaPropogation}%
	\end{figure}
	
	The proof of Lemma~\ref{L:key-bridge} proceeds in stages. We first define an explicit $\scrF(k, t)$-measurable event $E_{b, m}$ dependent on two parameters $m$ and $b$ that we will set later along with a random index $J$. Next, we show that conditionally on $E_{b, m}$, the probability in~\eqref{E:cond-J} can be bounded below in terms of $b$ and $m$. Finally we bound the probability of $E_{b, m}$, and then choose particular values of $b$ and $m$ to prove the lemma.
	
	\begin{proof}	
		Fix $k \in \N$ and $\ga \in [c \log \log k/ \log k, 2]$ for some large constant $c$. Exactly how large we need to take $c$ for the lemma to hold will be made clear in the proof.
		
		
		\textbf{Step 1: Defining $E_{b, m}$ and $J$.} \qquad Let $m \in \N$ be such that $m < k/4$, and let $b\in(0,k^{\gamma/4}/m]$. Set $a_i(s) = -\kappa i^{2/3} - s^2$ (an approximation to $\E \scrL_i(s)$). Let $E_{b, m}$ be the event where:
		\begin{enumerate}[label=(\roman*)]
			\item $\scrL_{2k + 1}(s) \le a_{2k + 1}(s) + bk^{-1/3}$ for all $s \in [0, t]$.
			\item $\scrL_\floor{3k/2}(s) \ge  a_\floor{3k/2}(s) - bk^{-1/3}$ for $s =0, t$.
			\item The graph $G_{2k}(t, 1, 5k^{-1/3-\ga/4})$ has no component of size greater than $m$.
		\end{enumerate}
		We define
		$$
		J = \max \Big\{j \le \frac{3k}2 : \min_{s\in\{0,t\}}\scrL_{j-1}(s) - \scrL_{j}(s) \ge 4k^{-1/3-\ga/4}  \Big\},
		$$
		setting $J = 1$ if the set above is empty.
		On the event $E_{b, m}$, condition (iii) implies that the above set is nonempty and that $J > k$ (this uses that $m < k/4$). To show that~\eqref{E:cond-J} is large on $E_{b, m}$, we will set up a stochastic domination argument.
		
		
		\textbf{Step 2: Separating out Airy points.} \qquad To begin, we would like to find points $h_i$ for $i=1,\ldots ,2k$ that dominate $\scrL_i(0)$ and have spacing at least $b\sqrt{t}$. See Figure~\ref{fig:hi} for a visual aid for this step. First define $h_{2k+1} = a_{2k +1}
		(0) + bk^{-1/3}$. For $i \le 2k + 1$, recursively define
		$$
		h_{i-1} =
		\max(h_i + b\sqrt{t} , \scrL_{i-1}(0)).
		$$
		We claim that
		\begin{equation}
		\label{E:J-bd}
		h_J \le \scrL_{J-1}(0) - 3k^{-1/3 - \ga/4}
		\end{equation}
		on the event $E_{b, m}$.
		To show this, set
		$$
		I = \{ i \in [J, 2k] : h_i = \scrL_i(0)\}.
		$$
		We first show that $I$ is nonempty with $ \sup I \ge 3k/2$.
		If $\sup I < 3k/2$, then
		\begin{equation}
		\label{E:h-3k}
		\begin{split}
		h_{\floor{3k/2}} &= a_{2k + 1}(0) + b k^{-1/3} + \lf(2k + 1 - \floor{3k/2}\rg)\sqrt{t} b \\
		&\le a_{2k + 1}(0) + k^{-1/3 + \ga/4} + k^{2/3 - \ga/4}.
		\end{split}
		\end{equation}
		Here the inequality uses that $t \le k^{-2/3 - \ga}, b \le k^{\ga/4}$ and that $2k + 1 - \floor{3k/2} \le k$ for $k \ge 3$. However, by (ii) in the definition of $E_{b, m}$, again using that $b \le k^{\ga/4}$, we have that
		$$
		h_{\floor{3k/2}} \ge a_{\floor{3k/2}}(0) - k^{\ga/4 - 1/3}.
		$$
		This bound contradicts~\eqref{E:h-3k} if we can show that
		\begin{equation}
		\label{E:kkga}
		k^{-1/3 + \ga/4} + k^{2/3 - \ga/4} + k^{\ga/4 - 1/3} \le \ka(2^{2/3}-(3/2)^{2/3})k^{2/3}.
		\end{equation}
		Since $\ga \le 2$, the left hand side above is bounded by $3k^{2/3 - \ga/4}$. The assumption that
		$$
		\ga > c \log (\log k)/ \log k
		$$
		then implies~\eqref{E:kkga} as long as $c$ is taken large enough (note that the weaker lower bound of $\ga > c /\log k$ would also suffice). This completes the proof that $I$ is nonempty and $\sup I \ge 3k/2$.
		
		
		Now, for $i \in [J,  \sup I]$, let
		$
		f(i) = \inf (I \cap [i, 2k]).
		$
		We claim that $f(i) - i < m$ for all $i \in [J, \sup I]$. If not, then $f(i) - i \ge m$ for some $i \in [J, \sup I]$. The recursive definition of $h_i$ then implies that
		$$
		h_i \ge m b\sqrt{t} + \scrL_{f(i)}(0) > \scrL_i(0).
		$$
		Since $b \le k^{\ga/4}/m$ and $t \le k^{-2/3 - \ga}$, this would imply that $i$ and $f(i)$ are in the same component of $G_{2k}(t, 1, 5k^{-1/3-\ga/4})$, contradicting condition (iii) on the set $E_{b, m}$.
		Therefore for all $i \in [J,  \sup I]$, $f(i) - i < m$ and so for such $i$, we have
		$$
		h_i - \scrL_i(0) < mb\sqrt{t} \le k^{-1/3-\gamma/4}.
		$$
		In the second inequality we have again used that $b \le k^{\ga/4}/m$ and $t \le k^{-2/3 - \ga}$.
		In particular, this bounds holds for $i=J$. By the definition of $J$, this implies~\eqref{E:J-bd}.
		
		
		We can define $h_i'$ in terms of the sequence $\{\scrL_i(t)\}$ analogously to how we defined $h_i$. This gives an analogous inequality to~\eqref{E:J-bd} for $h'_J$.
		
		\begin{figure}
			\centering
			\includegraphics[width=4cm]{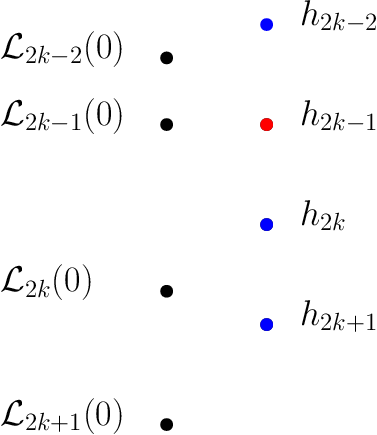}
			\caption{A visual aid for step 2. The points $h_i$ are defined recursively so that $h_i \ge \scrL_i(0)$ for all $i \ge 2k$, and so that the $h_i$ are well-spaced. In the above figure, all points have been moved up except $\scrL_{2k-1}(0)$. Hence $I \cap [2k-2, k] = \{2k-1\}$ and $f(2k-2) = f(2k-1) = 2k-1$.}%
			\label{fig:hi}
		\end{figure}
		
		\textbf{Step 3: A stochastic domination argument.} \qquad Now, by Lemma~\ref{L:monotone-gibbs}, on the event $E_{b, m}$ the distribution of $\{A_J(s) : s \in [0, t]\}$ conditional on $\scrF(k, t)$ is stochastically dominated by the top line $B_J$ of a system of Brownian bridges $(B_{J},\ldots,B_{2k})$ of variance $2$ on $[0, t]$ with endpoints $B_i(0) = h_i, B_i(t) =h'_{i}$, conditioned to avoid each other and stay above the boundary
		$$
		h_{2k + 1}(s) = a_{2k + 1}(s) + bk^{-1/3}.
		$$
		In particular, on $E_{b,m}$ we can bound the conditional probability in~\eqref{E:cond-J} below by
		\begin{equation}\label{E:eventabove}
		\p \Big(\min_{s\in [0, t]} L_{J-1}(s) -  B_{J}(s) \ge k^{-1/3 - \gamma/4} \; \Big | \; \scrF(k, t) \Big)
		\end{equation}
		Now the $B_i$ can be realized by repeated sampling of independent Brownian bridges until the avoidance conditions are satisfied. Let $(C_J, \dots, C_{2k})$ denote the first sample, and let $S$ be the event that it successfully satisfies all avoidance conditions. Then the event in~\eqref{E:eventabove} is implied by
		$$S \cap P, \qquad \text{ where } \quad P = \left\{\min_{s\in [0, t]} L_{J-1}(s) - C_{J}(s) \ge k^{-1/3 - \gamma/4}\right\}.
		$$
		This event is implied by the event $R$ where all bridges $C_i$ stay in a channel of width $b\sqrt{t}/2$ about the line from $(0, h_i)$ to $(t,h_i')$. To see this, first note that since $|h_i - h_{i+1}|, |h_i' - h_{i' + 1}| \ge b\sqrt{t}$ by construction, that $R \sset S$. Second, by~\eqref{E:J-bd}, the analogous inequality for $h_J'$, and the bounds on $b$ and $t$, we have
		$$
		h_J \le \scrL_{J-1}(0) - 3b\sqrt{t}, \qquad h_J' \le \scrL_{J-1}(t) - 3b\sqrt{t}.
		$$
		This implies that $R \sset P$ as well. The conditional probability of $R$ can be controlled a standard bound on the maximum of a Brownian bridge and a union bound. This yields that on the event $E_{b, m}$ we have
		\begin{equation}
		\label{E:ck-db}
		\p(R \;| \; \scrF(k, t)) \ge 1 - cke^{-db^2}
		\end{equation}
		for universal constants $c$ and $d$.  Hence the same inequality holds for the conditional probability in~\eqref{E:cond-J}.
		
		
		\textbf{Step 4: Bounding $\p(E_{b, m})$:} \qquad We bound $\p(E_{b, m})$ assuming that $b > \beta \log k$ for some constant $\beta$. With this lower bound on $b$, we can use Corollary~\ref{C:k-line-bd} to bound the probability that condition (i) holds and Corollary~\ref{C:point-locations} to bound probability of condition (ii). We use Proposition~\ref{P:delta-chains} for condition (iii) with $\ell = 1$ and $\de = 5k^{-\ga/4}$. These bounds yield
		$$
		\p(E_{b, m}) \ge 1 - ce^{-b/10} - [c(1 + \ga/4)m \log k]^mk^{1 - \frac{3\ga}{4}\floor{m/6}}
		$$
		for a constant $c$.
		
		
		Finally, taking $b = k^{\ga/6}, m = k^{\ga/12},$ and using that $\ga \in [c \log \log k/ \log k, 2]$ gives that the right hand side above is bounded below by
		$
		1 - e^{d \ga k^{\ga/12}}
		$
		for a constant $d$. Moreover, the lower bound on $\ga$ also implies that~\eqref{E:ck-db} (and hence~\eqref{E:cond-J}) is bounded below by $1 - e^{-dk^{\ga/4}}$. This completes the proof of Lemma~\ref{L:key-bridge}.
	\end{proof}
	
	We are now ready to prove the bridge representation theorem. The basic idea of the proof is to construct a line ensemble $\scrC^k$ on $\{1, \dots, 2k\} \X [0, t]$ that interpolates between $\scrB^k := \scrB^k(t, \ell, k^{-1/3 -\ga/4})$ and $\scrL$ as follows.
	\begin{itemize}[nosep]
		\item $\scrC^k$ and $\scrB^k$ will have the same law on the set $\{1, \dots, k\} \X [0, t]$.
		\item The total variation distance between the $\scrC^k$ and $\scrL|_{\{1, \dots, 2k\} \X [0, t]}$ will be bounded above by $\ell e^{-d \ga k^{\ga/12}}$.
	\end{itemize}
	The line ensemble $\scrC^k$ will be constructed by replacing the lines $\scrB^k_{k+1}, \dots, \scrB^k_{2k}$ with Brownian bridges $A_{i, j}$ constructed via Lemma~\ref{L:key-bridge}. We will then be able to estimate the total variation distance between $\scrC^k$ and $\scrL|_{\{1, \dots, 2k\} \X [0, t]}$ by estimating the probability that $\scrC^k$ intersects itself or the lower boundary $\scrL_{2k + 1}$.
	
	
	Throughout the proof, $b$ and $d$ will be positive constants that may change from line to line. We use the notation of Lemma~\ref{L:key-bridge}.
	\begin{proof}[Proof of Theorem~\ref{T:bridge-rep}]
		Set $\ga \in [c \log \log k/ \log k, 2]$ for a sufficiently large constant $c$. As with Lemma~\ref{L:key-bridge}, exactly how large we need to take $c$ will be made clear in the proof.
		Let $\scrH$ be the $\sig$-algebra generated by the set of random variables
		\begin{equation}
		\label{E:S-set}
		S = \{\scrL_i(s_j) : (i, j) \in \{1, \dots, 2k\} \X \{0, \dots, \ell\} \} \cup \{\scrL_{2k + 1}|_{[0, t]}\}.
		\end{equation}
		
		By Lemma~\ref{L:key-bridge} and a union bound, there exist random variables $J_1, \dots, J_{\ell} \in (k, 2k]$ such that the following holds. For $j \in \{1, \dots, \ell\}$, define a sequence of Brownian bridges $A_j = (A_{J_j, j}, \dots, A_{2k, j})$ on the interval $[s_{j-1}, s_j]$ as in the statement of Lemma~\ref{L:key-bridge} such that $A_{i, j}(s_{j-1}) = \scrL_i(s_{j-1})$ and $A_{i, j}(s_{j}) = \scrL_i(s_{j})$, and such that the bridges in $A_j$ are conditioned to avoid each other and the boundary $\scrL_{2k+1}$.  	
		Also let $L_j$ be the line with $L_j(s_{j-1}) = \scrL_{J_j - 1}(s_{j-1})$ and $L_j(s_{j}) = \scrL_{J_j - 1}(s_{j})$.
		Then with probability at least $1 - \ell e^{- d k^{\ga/12}}$, we have
		\begin{equation}
		\label{E:big-p}
		\begin{split}
		\p \lf(R | \; \scrH \rg) &\le \ell e^{- d \ga k^{\ga/4}} \qquad 	\text{ where } \qquad \\
		R &= \lf\{\exists j \in \{1, \dots, \ell\} : \min_{q \in [s_{j-1}, s_{j}]}L_j(q) - A_{J_j, j}(q) \le k^{-1/3 - \gamma/4} \rg\}.
		\end{split}
		\end{equation}	
		We now define the line ensemble $\scrC^k$. For each $j \in \{1, \dots, \ell\}$ and $i \in \{J_j, \dots, 2k\}$, set $C_{i, j} = A_{i, j}$. For $j \in \{1, \dots, \ell\}$ and $i \in \{1, \dots, J_j- 1\}$, set $C_{i, j} = B_{i, j}$, where the bridges $B_{i, j}$ are as in Definition~\ref{D:bridge-rep}. Concatenating the bridges $C_{i, j}$ together for $j \in \{1, \dots, \ell\}$ gives the $i$th line $\scrC^k_i$ in the line ensemble $\scrC^k$.

		
		Now, the conditional distributions of both $\scrL|_{\{1, \dots, 2k\} \X [0, t]}$ and $\scrC^k$ given $\scrH$ are the laws of $2k \ell$ independent Brownian bridges
		$$
		\{Z_{i, j} : (i, j) \in \{1, \dots, 2k\} \X \{1, \dots, \ell\} \}
		$$
		with endpoints $Z_{i, j}(s_{j-1}) = \scrL_i(s_{j-1})$ and $Z_{i, j}(s_{j}) = \scrL_i(s_{j})$, conditioned on two different sets of nonintersection events.
		
		
		Since the nonintersection conditions for $\scrL$ are stronger than those of $\scrC^k$, the total variation distance of their laws is simply the probability that $\scrC^k$ does not satisfy the nonintersection conditions for $\scrL$. This is
		\begin{equation}
		\label{E:P-Ck}
		\p(D), \;\; \text{ where } D = \{\text{$\scrC^k$ intersects itself or the lower boundary $\scrL_{2k + 1}$}\}.
		\end{equation}
		Because the bridges $A_{i, j}$ were conditioned both to be nonintersecting and not to intersect $\scrL_{2k+1}$, the only way that the event in~\eqref{E:P-Ck} can hold is if either
		\begin{itemize}[nosep]
			\item $B_{J_{j}-1, j}$ intersects $A_{J_j, j}$ for some $j \in \{1, \dots \ell\}$, or
			\item a pair of bridges $B_{i, j}$ and $B_{i+1, j}$ intersect for some $j \in \{1, \dots, \ell\}, i \in \{1, \dots, J_{j}-2\}$ where $(i, j)$ and $(i+1, j)$ are not connected by an edge in $G_{2k}(t, \ell, k^{-1/3 - \ga/4})$.
		\end{itemize}
		Therefore letting $E_{i,j}$ be the event where the bridge $B_{i, j}$ leaves a channel of width $k^{-1/3 - \ga/4}/2$ around the line between its endpoints, we have
		\begin{equation}
		\label{E:D-Rc}
		D \sset R^c \cup \bigcup_{j \in [1, \ell], i \in [1, J_{j}-1]} E_{i, j}.
		\end{equation}
		We now bound $\p(E_{i, j} \; | \; \scrH)$. Let $Q_{i, j}$ be the slope of the line segment connecting the points $(s_{j-1}, \scrL_i(s_{j-1}))$ and $(s_{j}, \scrL_i(s_j))$ and let $M$ be the size of the largest component of $G_{2k}(t, \ell, k^{-1/3 - \ga/4})$. Each of the bridges $B_{i, j}$ has slope $Q_{i, j}$ and is conditioned to avoid at most $M-1$ other bridges. Therefore for all $j \in [1, \ell]$ and $i \in [1, J_{j-1}]$, we can apply Proposition~\ref{P:bridge-bd} to the bridge $B_{i, j}$ on the interval $[s_{j-1}, s_j]$ to get that conditionally on $\scrH$, we have that
		\begin{equation}
		\label{E:Bijsr}
		|B_{i, j}(s + r) - B_{i, j}(s) - Q_{i, j} r| \le P_{i, j} \sqrt{r} \log^{1/2} \lf( \frac{2 t}{\ell r} \rg)
		\end{equation}
		for all $s, s + r \in [s_{j-1}, s_j]$
		for a random constant $P_{i, j}$ satisfying
		$$
		\p\lf( P_{i, j} \ge m \; | \; \scrH \rg) \le e^{c M - d m^2}.
		$$
		The right hand side of \eqref{E:Bijsr} is always bounded above by $P_{i, j} \sqrt{t/\ell}$, so $E_{i, j}$ implies the event $P_{i, j} \ge \frac{1}2 k^{-1/3 - \ga/4} \sqrt{\ell/t}$. Therefore we get that
		\begin{equation}
		\label{E:Eij}
		\p(E_{i, j} \;|\; \scrH) \le \exp\lf(b M - d k^{-2/3 -\ga/2}\ell/t \rg) \le \exp\lf(b M - d k^{\ga/2} \rg).
		\end{equation}
		Taking a union bound over $E_{i, j}$ in~\eqref{E:Eij}, on the $\scrH$-measurable event $M \le k^{\ga/12}$ we get that
		\begin{equation}
		\label{E:Eij-H-bd}
		\p\lf(\bigcup_{j \in [1, \ell], i \in [1, J_{j}-1]} E_{i, j} \;\; \bigg| \;\; \scrH \rg)  \le k\ell \exp\lf(b k^{\ga/12} - d k^{\ga/2} \rg) \le b \ell k \exp \lf(-d k^{\ga/2}\rg).
		\end{equation}
		In the second inequality, we have used that $\ga \ge c \log \log k/\log k$ to ensure that $k^{\ga/2} \gg k^{\ga/12}$ as $k \to \infty$.
		Now we bound $\p(M > k^{\ga/12})$ to ensure that \eqref{E:Eij-H-bd} holds with high probability.
		By Proposition~\ref{P:delta-chains}, for any $m > 0$ we have that
		$$
		\p(M > m) \le \ell [b(1 + \ga^3) m \log k]^m k^{1 -3(\gamma/4) \floor{m/6}}.
		$$
		In particular, setting $m = k^{\ga/12}$ we get that
		$$
		\p(M > k^{\ga/12}) \le \ell [b(1 + \ga^3) k^{\ga/12} \log k]^{k^{\ga/12}} k^{1 -3(\gamma/4) \floor{k^{\ga/12}/6}} \le \ell e^{- d \ga k^{\ga/12}}.
		$$
		For the second inequality we have used that $\ga \in [c \log \log k/\log k, 2]$, which ensures that $k^{\ga/12} \gg \log k \gg 1 + \ga^3$ as $k \to \infty$.
		
		
		Hence by combining~\eqref{E:big-p},~\eqref{E:D-Rc}, and~\eqref{E:Eij-H-bd}, we have
		$$
		\p(D \; | \; \scrH) \le \ell e^{-d \ga k^{\ga/12}} + b \ell k e^{-d k^{\ga/2}}
		$$
		with probability at least $1 - \ell e^{-d \ga k^{\ga/12}}$. Averaging over $\scrH$ shows that $\p(D) \le \ell e^{-d \ga k^{\ga/12}}$, completing the proof of Theorem~\ref{T:bridge-rep}.
	\end{proof}
	
	When using the bridge representation, it is important that we have a good understanding of the graph $G_{2k}(t, \ell, \de)$. In this vein, we conclude this section with an estimate on the density of nonisolated vertices in $G_{2k}(t, \ell, \de)$.
	
	
	While this next proposition is not used in the remainder of the paper, we include it anyways for an important application in~\cite{DOV} and to demonstrate how edge locations in the graph $G_{2k}(t, \ell, \de)$ can be controlled. Note that we have not tried to optimize the parameter bounds in Proposition~\ref{P:edge-spread}. In particular, the restriction on $m$ is rather artificial, and is chosen for ease of proof.
	
	\begin{prop}
		\label{P:edge-spread}
		Fix $\ga \in (0, 2]$, and let $k \in \nat, t \in (0, \infty), \ell \in \N$. Let the graph $G := G_{2k}(t, \ell, k^{-1/3 - \ga/4})$ be as in Definition~\ref{D:bridge-graph}. For each $j \in \{1, \dots, \ell\}$, let
		$$
		V_j = \{x \in \{1, \dots, 2k\} :  \deg_G(x, j) \ge 1\}.
		$$
		In other words, $V_j$ is the set of vertices in $G$ with second coordinate $j$ that are connected to at least one other vertex. Then for any $\al \in (0, 1]$ there exist constants $c_{\al}, d_{\al}$ such that for all  $k \in \N, i \in \{\floor{k^{\al}} + 1, \dots, 2k\},$ and $m \le k^{\al/2}$, we have that
		\begin{equation}
		\label{E:Vj}
		\p(|V_j \cap \{i - \floor{k^{\al}}, \dots, i \}| > mk^{\al - 3\ga/4}) \le c_{\al} e^{-d_{\al} m}.
		\end{equation}
		
	\end{prop}
	
	In the proof, $c', d > 0$ are constants that may change from line to line.
	
	\begin{proof}
		Fix $k_\al$ large enough so that for all $k \ge k_\al$ and $i \le 2k$, we have $k^{\al/2} \ge \beta \log (i + 1)$, where $\beta$ is as in Corollary~\ref{C:point-locations}. We first prove the proposition for $k < k_\al$. In this case, $m \le k_\al^{\al/2} \le k_\al^{1/2}$. Therefore we can always guarantee that the right hand side of~\eqref{E:Vj} is greater than $1$ by possibly modifying $c_\al, d_\al$.
		
		
		Now suppose that $k \ge k_\al$. Fix $i \in \{\floor{k^\al} + 1, \dots, 2k\}$, and let
		$$
		a_i = \ka i^{2/3} + k^{\al/2} i^{-1/3}, \qquad \ell_i = c a_i^{-2} k^{1 + \al}, \qquad I_i = [-a_i, -a_i + \ell_i],
		$$
		for a large constant $c$ to be determined in the proof. For any $s \in \R$ and $i \in \{\floor{k^\al} + 1, \dots, 2k\}$ we claim that
		\begin{equation}
		\label{E:Ri}
		\p(\{\scrL_{i - \floor{k^{\al}}}(s) + s^2, \dots, \scrL_i(s) + s^2\} \nsubseteq I_i) \le 2 e^{-k^{\al/2}/5}.
		\end{equation}
		By the ordering of the points $\scrL_j(s)$ and a union bound, the left side of \eqref{E:Ri} is bounded above by
		\begin{equation}
		\label{E:Ri'}
		\p(\{\scrL_{i - \floor{k^{\al}}}(s) + s^2 > - a_i + \ell_i) + \p(\scrL_i(s) + s^2 < a_i).
		\end{equation}
		To bound the two probabilities in \eqref{E:Ri'}, we will use Corollary~\ref{C:point-locations}. In order to fruitfully apply this corollary to bound the first probability above, we will need a lower bound on $\ell_i$. We will show that
		\begin{equation}
		\label{E:ell-i}
		\ell_i \ge a_i -\ka (i- \floor{k^\al})^{2/3} + k^{\al/2} (i- \floor{k^\al})^{-1/3}
		\end{equation}
		for all $i \ge \floor{k^\al} + 1$. First, since $i \ge k^\al \ge k^{\al/2}$, we have that $a_i \le (\ka + 1)i^{2/3}$, and so
		\begin{equation}
		\label{E:ellii}
		\ell_i = \frac{c k^{1 + \al}}{a_i^2} \ge \frac{c}{(\ka + 1)^2} k^{1 + \al} i^{-4/3} \ge \frac{c}{2(\ka + 1)^2} k^{\al} i^{-1/3}.
		\end{equation}
		The second inequality uses that $2k \ge i$. Now,
		\begin{align}
		\nonumber
		a_i -\ka (i- \floor{k^\al})^{2/3} + k^{\al/2} (i- \floor{k^\al})^{-1/3} &\le 2 k^{\al/2} (i- \floor{k^\al})^{-1/3} + \ka [i^{2/3} - (i- \floor{k^\al})^{2/3}] \\
		\nonumber
		&\le 2 k^{\al/2} (i- \floor{k^\al})^{-1/3} + \ka i^{-1/3} k^\al \\
		\label{E:ki13}
		&= k^{\al} i^{-1/3} \lf(\ka + \frac{2i^{1/3}}{k^{\al/2}(i- \floor{k^\al})^{1/3}} \rg).
		\end{align}
		Now, the term in the brackets in \eqref{E:ki13} is bounded above by a constant for all $k \in \N, i \ge \floor{k^{\al}} + 1$. Therefore the right side of \eqref{E:ki13} is bounded above by the right side of \eqref{E:ellii} for all $k \in \N, i \ge \floor{k^{\al}} + 1$ as long as $c$ is large enough, proving \eqref{E:ell-i}.
		
		
		Therefore \eqref{E:Ri'} is bounded above by
		\begin{align*}
		&\p\lf(\scrL_i(s) + s^2 < -k^{\al/2}i^{-1/3}  - \ka i^{2/3}\rg) \\
		+  &\p\lf(\scrL_{i - \floor{k^{\al}}}(s) + s^2 > -\ka (i - \floor{k^{\al}})^{2/3} + k^{\al/2}(i - \floor{k^{\al}})^{-1/3}\rg).
		\end{align*}
		Since $k^{\al/2} \ge \beta \log (i + 1)$, we can bound these probabilities above with Corollary~\ref{C:point-locations}, yielding \eqref{E:Ri}.
		
		
		Now using that $a_i^2 \ell_i = k^{1 + \al}$, by Proposition~\ref{P:mean-conc}, the number $L$ of $k^{-1/3-\ga/4}$-jammed points for the point process $\{\scrL_i(s) + s^2\}_{i \in \N}$ in the interval $I_i$ satisfies
		$$
		\p(L > mk^{\al - 3\ga/4}) \le c'e^{-dm}.
		$$
		By combining this with the bound in~\eqref{E:Ri}, we can bound the size of $V_j \cap \{i - \floor{k^{\al}}, \dots, i \})$ by noting that the number of noninsolated vertices can be bounded by the sum of the number of $\de$-jammed points in the sequences $\{\scrL_i(s_{j-1}) + s_{j-1}^2 \}_{i \in \N}$ and $\{\scrL_i(s_{j}) + s_j^2\}_{i \in \N}$.
	\end{proof}
	
	\section{A stronger modulus of continuity for parabolic Airy lines}\label{S:pr-bridge}
	
	In this section we give an application of the bridge representation to get a modulus of continuity bound on parabolic Airy lines that has a better $k$-dependence than Theorem~\ref{T:airy-mod}. Theorem~\ref{T:mod-cont-i} is an immediate corollary of our next theorem.

	\begin{theorem}
		\label{T:mod-cont}
		There exists a constant $d > 0$ such that for any $t > 0$, we have
		\begin{equation}
		\label{E:k-summa}
		\sum_{k \in \N} \p\lf( \sup_{s, s + r \in [0, t]} \frac{|\scrL_k(s) - \scrL_k(s + r)|}{\sqrt{r} \log^{1/2}(1 + r^{-1}) \log^d k} > 1 \rg) < \infty.
		\end{equation}
		Moreover, for any $t > 0$, we have that
		$$
		\sup_{k \in \N} \sup_{s, s + r \in [0, t]} \frac{|\scrL_k(s) - \scrL_k(s + r)|}{\sqrt{r} \log^{1/2}(1 + r^{-1}) \log^d k} < \infty \qquad \text{ almost surely }.
		$$
	\end{theorem}
	The second statement follows from the first by the Borel-Cantelli lemma and Theorem~\ref{T:airy-mod}. Note that if we remove the supremum over $k$, then the second statement is implied by Theorem~\ref{T:airy-mod}. However, if we simply apply that theorem and take a union bound, we need to put a factor of $\sqrt{k}$ in the denominator, rather than a power of $\log k$.
	
	
	On small time intervals (i.e. if we only take the supremum over $s, s + r \in [0, t]$ with $r < k^{-2/3 - \ga}$ for some fixed $\gamma > 0$), a variant of our proof will give the power $\log^{1/2} k$ instead of $\log^d(k)$, which is the same as what one would get for sequences of independent Brownian motions.
	
	
	Throughout the proof, $b$ and $d$ are large positive constants and $d_0$ is a small positive constant. These may change from line to line. The constant $b$ may depend on $t$, but $d$ and $d_0$ will not.
	
	\begin{proof}
		We will show that there exists a $d > 0$ and a $t$-dependent constant $b$ such that
		\begin{equation}
		\label{E:k-summa'}
		\sum_{k \in \N} \p\lf( \sup_{s, s + r \in [0, t]} \frac{|\scrL_k(s) - \scrL_k(s + r)|}{\sqrt{r} \log^{1/2}(1 + r^{-1}) \log^d k} > b \rg) < \infty
		\end{equation}
		This immediately implies \eqref{E:k-summa} with $d$ replaced by $d + 1$.
		Let $\scrB^k$ denote the bridge representation
		$$
		\scrB^k\lf(t, \ell_k, k^{-1/3 - \ga/4}\rg), \qquad \text{ where } \ell_k = \ceil{tk^{2/3 + \ga}}.
		$$
		with $\ga = c \log \log(k)/ \log (k)$. Here $c$ is a constant that can be chosen to be large enough so that each of the subsequent steps in the proof goes through.
		By Theorem~\ref{T:bridge-rep}, for all $k$ such that $c \log \log(k)/ \log (k) < 2$, we can couple the bridge representations $\scrB^k$ to $\scrL$ so that
		$$
		\p(\scrB^k|_{\{1, \dots, k\} \X [0, t]} \ne \scrL|_{\{1, \dots, k\} \X [0, t]}) \le \ceil{tk^{2/3 + \ga}}e^{- d_0\ga k^{\ga/12}}.
		$$
		By noting that $k^\ga = \log^c k$, as long as $c$ was chosen large enough, the right hand side above is bounded by $k^{-2}$ for all large enough $k$ and so
		$$
		\sum_{k \in \N} \p(\scrB^k|_{\{1, \dots, k\} \X [0, t]}  \ne \scrL|_{\{1, \dots, k\} \X [0, t]} ) < \infty.
		$$
		In particular, this means that it suffices to prove~\eqref{E:k-summa'} with the $k$th line of $\scrB^k$, denoted by $B_k$, in place of $\scrL_k$. Specifically, it is enough to show that for large enough $b$, we have
		\begin{equation}
		\label{E:B-bd}
		\p \lf(\sup_{s, s + r \in [0, t]} \frac{|B_k(s) - B_k(s + r)|}{\sqrt{r} \log^{1/2}(1 + r^{-1}) \log^d k} > b \rg) = O(k^{-2}).
		\end{equation}
		Define
		$$
		s_i = it/\ell_k \quad \mathfor \quad i \in \lf\{0, \dots, \ell_k \rg\}.
		$$
		Let $L_k$ denote the function which is equal to $B_k$ at the grid points $s_i$ and affine on each of the intervals $[s_{i-1}, s_i]$ and set $X_k = B_k - L_k$. By the triangle inequality and a union bound, it suffices to prove the bound in~\eqref{E:B-bd} for $L_k$ and $X_k$ separately. Letting $a_k(s) = \ka k^{2/3} + s^2$, Corollary~\ref{C:point-locations} and a union bound implies that there exists $c' > 0$ such that for all large enough $b$, we have
		\begin{equation}
		\label{E:Plf}
		\p\lf( |B_k(s_i) + a_k(s_i)| > b \log (k) k^{-1/3} \text{ for some } i \in \lf\{0, \dots, \ell_k \rg\} \rg) \le t e^{(c'-b)\log k/5}
		\end{equation}
		for all $k \ge 2$. For every $t$, there exists a $b_0$ such that the right hand side of \eqref{E:Plf} is bounded above by $k^{-2}$ for all $b \ge b_0$ and $k \ge 2$. 
		
		Now, on the event in \eqref{E:Plf} with $b = b_0$, the slope of each of the linear pieces of $L$ is bounded above by
		\begin{equation}
		\label{E:2blogk}
		2b_0 \log(k) k^{-1/3} \ell_k/t + 2t \le 2b_1 k^{1/3 + \ga}  \log k,
		\end{equation}
		where $b_1 \ge b_0$ is another $t$-dependent constant. Here the additive factor of $2t$ comes from the parabolic shape.
		Therefore on the event in \eqref{E:Plf}, for $s, s + r \in [0, t]$ with $r \le k^{-2/3 - 2\ga}$ we have
		\begin{equation}
		\label{E:Lksoft}
		|L_k(s) - L_k(s + r)| \le 2b_1 \sqrt{r} \log k.
		\end{equation}
		Now, again on the event in \eqref{E:Plf} with $b = b_0$, we have that
		\begin{equation}
		\label{E:Lkparabolic}
		|L_k(s) + a_k(s)| \le b_0 k^{-1/3} \log k + \frac{t^2}{4 \ell_k^2} \le b_2 k^{-1/3} \log k. 
		\end{equation}
		for all $s \in [0, t]$, where $b_2 \ge b_0$ is a $t$-dependent constant. Therefore for all $r$, we have
		$$
		|L_k(s) - L_k(s+r)| \le 2tr + 2 b_2 k^{-1/3} \log k.
		$$
		where the factor of $2tr$ again comes from the parabolic shape. Since $r \le t$, we have $2 t r \le 2t^{3/2} r^{1/2}$. Therefore for $r \ge k^{-2/3 - 2\ga} = k^{-2/3} \log^{-2c}(k)$ and $k \ge 2$, we have
		\begin{equation}
		\label{E:Lkhard}
		|L_k(s) - L_k(s+r)| \le (2 t^{3/2}  + 2 b_2 \log^{c + 1}(k)) \sqrt{r} \le b_3 \log^{c + 1}(k) \sqrt{r}
		\end{equation}
		where $b_3$ is another $t$-dependent constant. Combining \eqref{E:Lkhard}, \eqref{E:Lksoft}, and the $O(k^{-2})$ tail bound on \eqref{E:Plf} for $b \ge b_0$ implies the bound in~\eqref{E:B-bd} for $L_k$.
		
		
		To get a modulus of continuity bound on $X_k$, we use the modulus of continuity bounds on nonintersecting Brownian bridges established in Proposition~\ref{P:bridge-bd}. First we need to bound the size of components in the underlying graph $G^k = G_{2k}(t, \ell_k, k^{-1/3 - \ga/4})$ that gives rise to the bridge representation $\scrB^k$. Letting $M^k$ be the size of the largest component in $G^k$, applying Proposition~\ref{P:delta-chains} with $\ell = \ceil{t k^{2/3 + \ga}}$ and $\de = k^{-\ga/4}$ gives that for every $m > 0$ and $k \in \N$, we have
		$$
		\p(M_k > m) \le k^2 [b m \log(k)]^m [\log(k)]^{-\frac{3c}4\floor{m/6}}.
		$$
		To get this bound we have again used the observation that $k^\ga = \log^c k$, as well as the crude bound that $\ceil{t k^{2/3 + \ga}} \le k$ for large enough $k$. In particular, setting $m = \log^{c/16} k$, as long as $c$ is large enough we get that
		\begin{equation}
		\label{E:Mk}
		\p(M_k > \log^{c/16} k) \le k^{-2}
		\end{equation}
		for all large enough $k$. Let $F_k$ denote the event where $M_k \le \log^{c/16} k$. Proposition~\ref{P:bridge-bd} applied to the bridge $B_k$ on the interval $[s_i, s_{i+1}]$ implies that for every $m > 0$ and every $i \in \{0, \dots, \ell_k \}$, we have
		$$
		\p \lf( \sup_{s, s + r \in [s_i, s_{i+1}]} \frac{|X_k(s) - X_k(s + r)|}{\sqrt{r} \log^{1/2}(2\ell_k /r) } > m \; \bigg| \;\; F_k \rg) \le e^{d \log^{c/16}(k) - d_0 m^2}.
		$$
		Taking $m = b \log^{c/32} k$ in the above expression for a large $b$ and simplifying, we get that
		$$
		\p \lf( \sup_{s, s + r \in [s_i, s_{i+1}]} \frac{|X_k(s) - X_k(s + r)|}{\sqrt{r} \log^{1/2}(1 + 1/r) \log^{1 + c/32}(k) } > b \; \bigg| \;\; F_k \rg) \le k^{-d_0 b}.
		$$
		Taking a union bound over the intervals $[s_i, s_{i+1}]$ and using that $X_k = 0$ at the grid points $s_i$ then yields
		\begin{equation*}
		\p \lf( \sup_{s, s + r \in [0, 1]} \frac{|X_k(s) - X_k(s + r)|}{ \sqrt{r} \log^{1/2}(1 + 1/r) \log^{1 + c/32}(k) } > b \; \bigg| \;\; F_k \rg) \le \ell_k k^{- d_0 b}.
		\end{equation*}
		Combining this with~\eqref{E:Mk} gives the bound in~\eqref{E:B-bd} for $X_k$, completing the proof of \eqref{E:k-summa'}.
	\end{proof}
	
	\noindent {\bf Acknowledgments.}
	B.V. thanks Janosch Ortmann for several useful discussions.
	
	\bibliographystyle{acm}
	\bibliography{ALEcitations3}

\begin{thebibliography}{10}

\bibitem{abramowitz1972handbook}
{\sc Abramowitz, M., and Stegun, I.~A.}
\newblock {\em Handbook of mathematical functions: with formulas, graphs, and
  mathematical tables}, vol.~55.
\newblock Dover publications New York, 1972.

\bibitem{adler2005pdes}
{\sc Adler, M., and Van~Moerbeke, P.}
\newblock {PDE}s for the joint distributions of the {D}yson, {A}iry and sine
  processes.
\newblock {\em The Annals of Probability 33}, 4 (2005), 1326--1361.

\bibitem{anderson2010introduction}
{\sc Anderson, G.~W., Guionnet, A., and Zeitouni, O.}
\newblock {\em An introduction to random matrices}, vol.~118.
\newblock Cambridge university press, 2010.

\bibitem{basu2018nonexistence}
{\sc Basu, R., Hoffman, C., and Sly, A.}
\newblock Nonexistence of bigeodesics in integrable models of last passage
  percolation.
\newblock {\em arXiv preprint arXiv:1811.04908\/} (2018).

\bibitem{basu2017coal}
{\sc Basu, R., Sarkar, S., and Sly, A.}
\newblock Coalescence of geodesics in exactly solvable models of last passage
  percolation.
\newblock {\em arXiv preprint arXiv:1704.05219\/} (2017).

\bibitem{basu2017invariant}
{\sc Basu, R., Sarkar, S., and Sly, A.}
\newblock Invariant measures for {TASEP} with a slow bond.
\newblock {\em arXiv preprint arXiv:1704.07799\/} (2017).

\bibitem{basu2014last}
{\sc Basu, R., Sidoravicius, V., and Sly, A.}
\newblock Last passage percolation with a defect line and the solution of the
  slow bond problem.
\newblock {\em arXiv preprint arXiv:1408.3464\/} (2014).

\bibitem{calvert2019brownian}
{\sc Calvert, J., Hammond, A., and Hegde, M.}
\newblock Brownian structure in the {KPZ} fixed point.
\newblock {\em arXiv preprint arXiv:1912.00992\/} (2019).

\bibitem{CH}
{\sc Corwin, I., and Hammond, A.}
\newblock {B}rownian {G}ibbs property for {A}iry line ensembles.
\newblock {\em Inventiones mathematicae 195}, 2 (2014), 441--508.

\bibitem{corwin2016kpz}
{\sc Corwin, I., and Hammond, A.}
\newblock {KPZ} line ensemble.
\newblock {\em Probability Theory and Related Fields 166}, 1-2 (2016), 67--185.

\bibitem{corwin2014ergodicity}
{\sc Corwin, I., and Sun, X.}
\newblock Ergodicity of the {A}iry line ensemble.
\newblock {\em Electronic Communications in Probability 19\/} (2014).

\bibitem{DOV}
{\sc Dauvergne, D., Ortmann, J., and Vir{\'a}g, B.}
\newblock The directed landscape.
\newblock {\em arXiv preprint arXiv:1812.00309\/} (2018).

\bibitem{DNV}
{\sc Dauvergne, D., and Vir{\'a}g, B.}
\newblock The scaling limit of the longest increasing subsequence.
\newblock {\em In preparation\/} (2020).

\bibitem{dimitrov2020characterization}
{\sc Dimitrov, E., and Matetski, K.}
\newblock Characterization of brownian gibbsian line ensembles.
\newblock {\em arXiv preprint arXiv:2002.00684\/} (2020).

\bibitem{dyson1962brownian}
{\sc Dyson, F.~J.}
\newblock A brownian-motion model for the eigenvalues of a random matrix.
\newblock {\em Journal of Mathematical Physics 3}, 6 (1962), 1191--1198.

\bibitem{georgiou2017geodesics}
{\sc Georgiou, N., Rassoul-Agha, F., and Sepp{\"a}l{\"a}inen, T.}
\newblock Geodesics and the competition interface for the corner growth model.
\newblock {\em Probability Theory and Related Fields 169}, 1-2 (2017),
  223--255.

\bibitem{grabiner1999brownian}
{\sc Grabiner, D.~J.}
\newblock Brownian motion in a {W}eyl chamber, non-colliding particles, and
  random matrices.
\newblock {\em Annales de l'Institut Henri Poincare (B) Probability and
  Statistics 35}, 2 (1999), 177--204.

\bibitem{gustavsson2005gaussian}
{\sc Gustavsson, J.}
\newblock Gaussian fluctuations of eigenvalues in the {GUE}.
\newblock {\em Annales de l'Institut Henri Poincare (B) Probability and
  Statistics 41}, 2 (2005), 151--178.

\bibitem{hammond2016brownian}
{\sc Hammond, A.}
\newblock {B}rownian regularity for the {A}iry line ensemble, and multi-polymer
  watermelons in {B}rownian last passage percolation.
\newblock {\em arXiv preprint arXiv:1609.02971\/} (2016).

\bibitem{hammond2019modulus}
{\sc Hammond, A.}
\newblock Modulus of continuity of polymer weight profiles in brownian last
  passage percolation.
\newblock {\em The Annals of Probability 47}, 6 (2019), 3911--3962.

\bibitem{hammond2019patchwork}
{\sc Hammond, A.}
\newblock A patchwork quilt sewn from brownian fabric: Regularity of polymer
  weight profiles in brownian last passage percolation.
\newblock {\em Forum of Mathematics, Pi 7\/} (2019).

\bibitem{hammond2020exponents}
{\sc Hammond, A.}
\newblock Exponents governing the rarity of disjoint polymers in brownian last
  passage percolation.
\newblock {\em Proceedings of the London Mathematical Society 120}, 3 (2020),
  370--433.

\bibitem{hammond2020modulus}
{\sc Hammond, A., and Sarkar, S.}
\newblock Modulus of continuity for polymer fluctuations and weight profiles in
  poissonian last passage percolation.
\newblock {\em Electronic Journal of Probability 25\/} (2020).

\bibitem{horn2012matrix}
{\sc Horn, R.~A., and Johnson, C.~R.}
\newblock {\em Matrix analysis}.
\newblock Cambridge university press, 2012.

\bibitem{hough2009zeros}
{\sc Hough, J.~B., Krishnapur, M., Peres, Y., and Virag, B.}
\newblock {\em Zeros of Gaussian analytic functions and determinantal point
  processes}, vol.~51.
\newblock American Mathematical Soc., 2009.

\bibitem{kallenberg2006foundations}
{\sc Kallenberg, O.}
\newblock {\em Foundations of modern probability}.
\newblock Springer Science \& Business Media, 2006.

\bibitem{ledoux2007deviation}
{\sc Ledoux, M.}
\newblock Deviation inequalities on largest eigenvalues.
\newblock In {\em Geometric aspects of functional analysis}. Springer, 2007,
  pp.~167--219.

\bibitem{ledoux2010small}
{\sc Ledoux, M., and Rider, B.}
\newblock Small deviations for beta ensembles.
\newblock {\em Electronic Journal of Probability 15\/} (2010), 1319--1343.

\bibitem{o2002random}
{\sc O'Connell, N.}
\newblock Random matrices, non-colliding processes and queues.
\newblock {\em S{\'e}minaire de probabilit{\'e}s de Strasbourg 36\/} (2002),
  165--182.

\bibitem{pimentel2016duality}
{\sc Pimentel, L.~P.}
\newblock Duality between coalescence times and exit points in last-passage
  percolation models.
\newblock {\em The Annals of Probability 44}, 5 (2016), 3187--3206.

\bibitem{prahofer2002scale}
{\sc Pr{\"a}hofer, M., and Spohn, H.}
\newblock Scale invariance of the {PNG} droplet and the {A}iry process.
\newblock {\em Journal of statistical physics 108}, 5-6 (2002), 1071--1106.

\bibitem{soshnikov2000gaussian}
{\sc Soshnikov, A.~B.}
\newblock Gaussian fluctuation for the number of particles in {A}iry, {B}essel,
  sine, and other determinantal random point fields.
\newblock {\em Journal of Statistical Physics 100}, 3-4 (2000), 491--522.

\end{thebibliography}

\end{document}